\RequirePackage{fix-cm}

\documentclass[smallextended]{svjour3}

\smartqed  
\usepackage{graphicx}

	\usepackage{amsmath}
\usepackage{amssymb}
\usepackage{float}

\usepackage{placeins}
\usepackage{array}
\usepackage{multirow}
\usepackage{hhline}
\usepackage{verbatim}
\usepackage[letterpaper,left=1.5in,right=1in,top=1in,bottom=1in]{geometry}
\usepackage{amsfonts}
\usepackage{enumerate}
\usepackage{graphicx}
\usepackage{tabu}
\usepackage{mathrsfs}

\usepackage{tikz}
\usepackage{pifont}
\usetikzlibrary{positioning}
\usepackage{ marvosym }
\usepackage{verbatim}
\usepackage[mathscr]{euscript}
\usepackage{tabu}
\usetikzlibrary{arrows}
\usetikzlibrary{shapes,snakes}
\newtheorem{exmp}{Example}[section]
\newcommand{\stage}{s}

\usepackage{setspace,color}
\newtheorem{observation}{Observation}
\newcommand{\dollarsign}{L}
\newcommand\widebar[1]{\mathop{\overline{#1}}}

\begin{document}

\title{Counting Path Configurations in Parallel Diffusion}

\titlerunning{Path Configurations}        

\author{Todd Mullen (Dalhousie University)         \\
        Richard Nowakowski (Dalhousie University)   \\
        Danielle Cox (Mount Saint Vincent University)
}

\authorrunning{Mullen, Nowakowski, and Cox}

\institute{T. Mullen \at              
              \email{Todd.Mullen@Dal.ca} \at
              ORCID: 0000-0003-2818-2734
           \and
           D. Cox \at
              \email{Danielle.Cox@MSVU.ca}
           \and
           R. Nowakowski \at
           	  \email{R.Nowakowski@Dal.ca}   
}

\date{Received: date / Accepted: date}

\maketitle

\begin{abstract}
Parallel Diffusion is a variant of Chip-Firing introduced in 2018 by Duffy et al. In Parallel Diffusion, chips move from places of high concentration to places of low concentration through a discrete-time process. At each time step, every vertex sends a chip to each of its poorer neighbours, allowing for some vertices to perhaps fall into debt (represented by negative stack sizes). In their recent paper, Long and Narayanan proved a conjecture from the original paper by Duffy et al. that every Parallel Diffusion process eventually, after some pre-period, exhibits periodic behaviour. 
With this result, we are now able to count the number of these periods that exist up to a definition of isomorphism. We determine a recurrence relation for calculating this number for a path of any length. If $T_n$ is the number of configurations with period length 2 that can exist on $P_n$ up to isomorphism and $n$ is an integer greater than 4, we conclude that $T_n = 3T_{n-1} + 2T_{n-2} + T_{n-3} - T_{n-4}$.
\end{abstract}

\newpage

\begin{Large}

\textbf{Declarations}

\end{Large}

\vspace{1cm}

\noindent \textbf{Funding:} This research was supported by the Natural Sciences and Engineering Research Council of Canada

\noindent \textbf{Conflicts of interest/Competing interests:} Not applicable

\noindent \textbf{Availability of data and material:} Not applicable

\noindent \textbf{Code availability:} Not applicable

\newpage

\section{Introduction}
\label{intro}
Introduced by Duffy et al. \cite{duffy}, Parallel Diffusion is a process defined on a simple finite graph, $G$, in which at every time step, chips are diffused throughout the graph following specific rules. Each vertex is assigned a \textit{stack size} which is an integral number that represents the number of chips a vertex has.  An assignment of stack sizes to the vertices of a graph $G$ is referred to as a \textit{configuration}, denoted $C=\{(v,|v|^{C}): v\in V(G)\}$, where $|v|^C$ is the stack size of $v$ in $C$. We omit the superscript when the configuration is clear. At each time step, the chips are redistributed via the following rules: If a vertex is adjacent to a vertex with fewer chips, it takes a chip from its stack and adds it to the stack of the poorer vertex. This creates a new configuration (see Figure~\ref{fig:parex}). We call this action the \textit{firing} of a vertex. At each time step of the diffusion process, every vertex fires simultaneously. Note that when a vertex with no poorer neighbours fires, it does not send any chips.

Long and Narayanan showed this process to be periodic \cite{long}, with every configuration eventually (after some number of steps) leading either to a single configuration in which every stack size is equal (period length of 1) or a pair of configurations which yield each other (period length of 2). As a byproduct of Long and Narayanan's work, we can now count the configurations that can exist on a given graph. Clearly, the number of configurations that can exist on a given graph is infinite because there are infinitely many integers. But what if we only were interested in counting those configurations that yield each other, described by Long and Narayanan? In this paper, we use Long and Narayanan's result that every period is length 1 or 2 to count the number of different configurations that can exist on a path up to a definition of isomorphism. Our method will involve viewing the transfer of chips as a mixed graph and excluding every mixed graph on $P_n$ that cannot possibly represent the flow of chips within the period. From there, we count the number of period configurations that exist on $P_n$ for each of the mixed graphs that were not excluded.   

A vertex $v$ is said to be \textit{richer} than another vertex $u$ in configuration $C$ if $|v|^C > |u|^C$. In this instance, $u$ is said to be \textit{poorer} than $v$ in $C$. If $|v|^C<0$, we say $v$ is \textit{in debt in $C$}.

We are interested in counting the number of configurations, $T_n$, on $P_n$, $n \geq 1$ (up to a definition of isomorphism). We will show, in Theorem~\ref{thm:fn}, that $T_n$ can be calculated for all $n \geq 5$ by the recurrence relation $T_n = 3T_{n-1} + 2T_{n-2} + T_{n-3} - T_{n-4}$. This relation has an asymptotic growth rate of roughly 3.6096 (Corollary~\ref{cor:pathconfigurationasymptotic}).

\begin{figure}
	\[
	\begin{tikzpicture}[-,-=stealth', auto,node distance=1.5cm,
	thick,scale=0.6, main node/.style={scale=0.6,circle,draw}]
	
	\node[main node, label={[red]90:0}] (1) 					    {$v_1$};					 
	\node[main node, label={[red]90:2}] (2)  [right of=1]        {$v_2$};
	\node[main node, label={[red]90:0}] (3)  [right of=2]        {$v_3$};  
	\node[main node, label={[red]90:4}] (4) 	[right of=3]        {$v_4$};					 
	\node[main node, label={[red]90:1}] (5)  [right of=4]        {$v_5$}; 
	
	\path[every node/.style={font=\sffamily\small}]		 
	
	(1) edge node [] {} (2)			
	(2) edge node [] {} (3)
	(3) edge node [] {} (4)
	(4) edge node [] {} (5);

	\end{tikzpicture} \]
	Initial firing
	\FloatBarrier

	\[
	\begin{tikzpicture}[-,-=stealth', auto,node distance=1.5cm,
	thick,scale=0.6, main node/.style={scale=0.6,circle,draw}]
	
	\node[main node, label={[red]90:1}] (1) 					    {$v_1$};					 
	\node[main node, label={[red]90:0}] (2)  [right of=1]        {$v_2$};
	\node[main node, label={[red]90:2}] (3)  [right of=2]        {$v_3$};  
	\node[main node, label={[red]90:2}] (4) 	[right of=3]        {$v_4$};					 
	\node[main node, label={[red]90:2}] (5)  [right of=4]        {$v_5$}; 
	
	\path[every node/.style={font=\sffamily\small}]		 
	
	(1) edge node [] {} (2)			
	(2) edge node [] {} (3)
	(3) edge node [] {} (4)
	(4) edge node [] {} (5);

	\end{tikzpicture} \]
	Firing at step 1
	\FloatBarrier

	\[
	\begin{tikzpicture}[-,-=stealth', auto,node distance=1.5cm,
	thick,scale=0.6, main node/.style={scale=0.6,circle,draw}]
	
	\node[main node, label={[red]90:0}] (1) 					    {$v_1$};					 
	\node[main node, label={[red]90:2}] (2)  [right of=1]        {$v_2$};
	\node[main node, label={[red]90:1}] (3)  [right of=2]        {$v_3$};  
	\node[main node, label={[red]90:2}] (4) 	[right of=3]        {$v_4$};					 
	\node[main node, label={[red]90:2}] (5)  [right of=4]        {$v_5$}; 
	
	\path[every node/.style={font=\sffamily\small}]		 
	
	(1) edge node [] {} (2)			
	(2) edge node [] {} (3)
	(3) edge node [] {} (4)
	(4) edge node [] {} (5);

	\end{tikzpicture} \]
	Firing at step 2
	\FloatBarrier

	\[
	\begin{tikzpicture}[-,-=stealth', auto,node distance=1.5cm,
	thick,scale=0.6, main node/.style={scale=0.6,circle,draw}]
	
	\node[main node, label={[red]90:1}] (1) 					    {$v_1$};					 
	\node[main node, label={[red]90:0}] (2)  [right of=1]        {$v_2$};
	\node[main node, label={[red]90:3}] (3)  [right of=2]        {$v_3$};  
	\node[main node, label={[red]90:1}] (4) 	[right of=3]        {$v_4$};					 
	\node[main node, label={[red]90:2}] (5)  [right of=4]        {$v_5$}; 
	
	\path[every node/.style={font=\sffamily\small}]		 
	
	(1) edge node [] {} (2)			
	(2) edge node [] {} (3)
	(3) edge node [] {} (4)
	(4) edge node [] {} (5);

	\end{tikzpicture} \]
	Firing at step 3
	\FloatBarrier

	\[
	\begin{tikzpicture}[-,-=stealth', auto,node distance=1.5cm,
	thick,scale=0.6, main node/.style={scale=0.6,circle,draw}]
	
	\node[main node, label={[red]90:0}] (1) 					    {$v_1$};					 
	\node[main node, label={[red]90:2}] (2)  [right of=1]        {$v_2$};
	\node[main node, label={[red]90:1}] (3)  [right of=2]        {$v_3$};  
	\node[main node, label={[red]90:3}] (4) 	[right of=3]        {$v_4$};					 
	\node[main node, label={[red]90:1}] (5)  [right of=4]        {$v_5$}; 
	
	\path[every node/.style={font=\sffamily\small}]		 
	
	(1) edge node [] {} (2)			
	(2) edge node [] {} (3)
	(3) edge node [] {} (4)
	(4) edge node [] {} (5);

	\end{tikzpicture} \]
	Firing at step 4
	\FloatBarrier

	\[
	\begin{tikzpicture}[-,-=stealth', auto,node distance=1.5cm,
	thick,scale=0.6, main node/.style={scale=0.6,circle,draw}]
	
	\node[main node, label={[red]90:1}] (1) 					    {$v_1$};					 
	\node[main node, label={[red]90:0}] (2)  [right of=1]        {$v_2$};
	\node[main node, label={[red]90:3}] (3)  [right of=2]        {$v_3$};  
	\node[main node, label={[red]90:1}] (4) 	[right of=3]        {$v_4$};					 
	\node[main node, label={[red]90:2}] (5)  [right of=4]        {$v_5$}; 
	
	\path[every node/.style={font=\sffamily\small}]		 
	
	(1) edge node [] {} (2)			
	(2) edge node [] {} (3)
	(3) edge node [] {} (4)
	(4) edge node [] {} (5);

	\end{tikzpicture} \]
	Firing at step 5
	\FloatBarrier

	\[
	\begin{tikzpicture}[-,-=stealth', auto,node distance=1.5cm,
	thick,scale=0.6, main node/.style={scale=0.6,circle,draw}]
	
	\node[main node, label={[red]90:0}] (1) 					    {$v_1$};					 
	\node[main node, label={[red]90:2}] (2)  [right of=1]        {$v_2$};
	\node[main node, label={[red]90:1}] (3)  [right of=2]        {$v_3$};  
	\node[main node, label={[red]90:3}] (4) 	[right of=3]        {$v_4$};					 
	\node[main node, label={[red]90:1}] (5)  [right of=4]        {$v_5$}; 
	
	\path[every node/.style={font=\sffamily\small}]		 
	
	(1) edge node [] {} (2)			
	(2) edge node [] {} (3)
	(3) edge node [] {} (4)
	(4) edge node [] {} (5);

	\end{tikzpicture} \]
	
	\FloatBarrier
	\caption{Several steps in a Parallel Diffusion game on $P_5$}
	\label{fig:parex}
\end{figure}
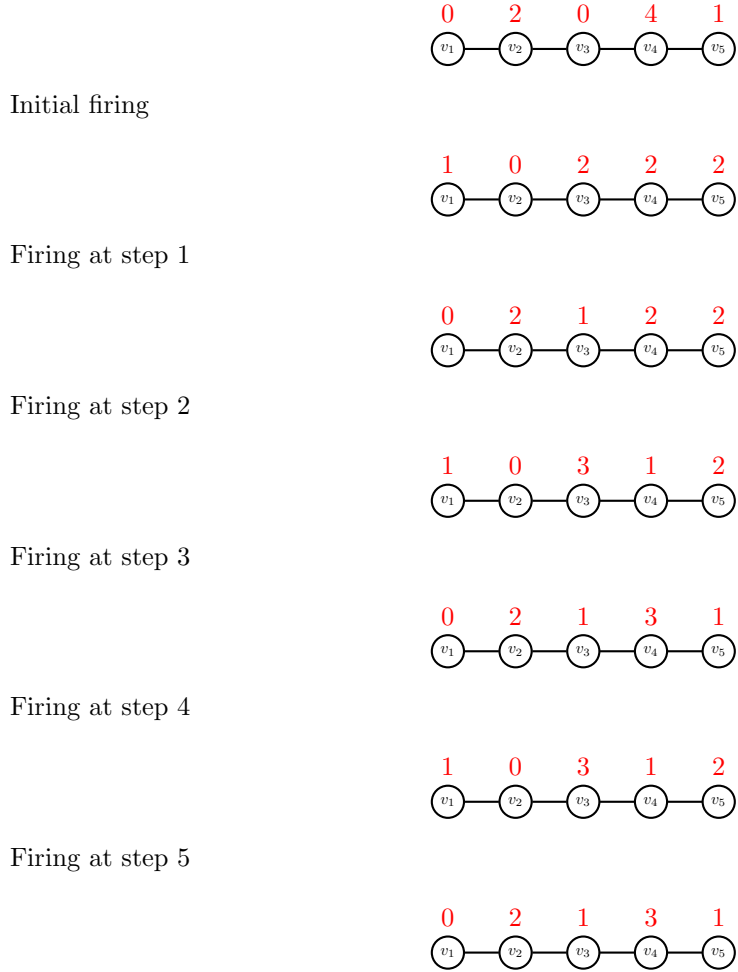
\FloatBarrier

Unlike some other chip-firing processes like the original Chip-Firing game \cite{lovasz} and Brushing \cite{pralat}, in Parallel Diffusion it is possible for a stack size to initially be positive but to become negative as time goes on. For example if some vertex $v$ with a stack size of $n$, $n \in \mathbb{N}$, is adjacent to $n+1$ vertices, each of which having a stack size of $0$, then after firing, $v$ would have at stack size of $-1$. However in \cite{duffy}, it was shown that Parallel Diffusion is such that an addition of some constant $k$, $k \in \mathbb{Z}$, to each stack size will have no effect on determining when and if a chip will move from one vertex to another. So if one wanted to view diffusion as a process in which stack sizes are never negative, one would only need to add a sufficient constant $k$, $k \in \mathbb{Z}$, to each stack size. Some results pertaining to locating an appropriate $k$ value for any given graph can be found in ``Uniform Bounds for Non-Negativity of the Diffusion Game" by Carlotti and Herrman, arXiv:1805.05932v1.

We begin with some necessary terminology.

Let $G$ be a finite simple undirected graph with vertex set $V(G)$ and edge set $E(G)$. Let $A \subseteq E(G)$. A \textit{graph orientation} of a graph $G$ is a mixed graph obtained from $G$ by choosing an orientation ($x \to y$ or $y \to x$) for each edge $xy$ in $A$. We refer to the edges that are in $E(G)$ \textbackslash $A$ as \textit{flat}. We refer to the assignment of either $x \to y$, $y \to x$, or flat to an edge $xy$ as $xy$'s \textit{edge orientation}.

Let $R$ be a graph orientation of a graph $G$. A \textit{suborientation} $R'$ of $R$ is a graph orientation of some subgraph $G'$ of $G$ such that every edge $xy$ in $G'$ is assigned the same edge orientation as in $R$.

\begin{figure} [h] 
	\centering	
	
	\begin{tikzpicture}[-,-=stealth', auto,node distance=1.5cm,
	thick,scale=0.6, main node/.style={scale=0.6,circle,draw}]
	
	\node[draw=none, fill=none]         (50)                       {$C_0$};
	\node[main node, label={[red]90:6}] (1)   [right=2cm of 50]    {$v_1$};	
	\node[main node, label={[red]90:2}] (2)   [right of=1]         {$v_2$};
	\node[main node, label={[red]270:5}] (3)  [below of=1]         {$v_3$};  
	\node[main node, label={[red]270:2}] (4)  [right of=3]         {$v_4$};			
	\node[main node, label={[red]270:6}] (5)  [right of=4]         {$v_5$}; 
	\node[draw=none, fill=none]         (51)  [below =2.5cm of 50] {$C_1$};
	\node[main node, label={[red]90:3}] (6)   [below=1.5cm of 3]   {$v_1$};
	\node[main node, label={[red]90:4}] (7)   [right of=6]         {$v_2$};
	\node[main node, label={[red]270:4}] (8)  [below of=6]         {$v_3$};  
	\node[main node, label={[red]270:5}] (9)  [right of=8]         {$v_4$};	
	\node[main node, label={[red]270:5}] (10) [right of=9]         {$v_5$}; 	
	
	\path   (7) edge (8)  (10) edge (9);
	
	\draw [->] (1) edge (3) (1) edge (2) (1) edge (4) (3) edge (2) (3) edge (4) (5) edge (4)  (7) edge (6) (9) edge (6) (8) edge (6)  (9) edge (8);

	\end{tikzpicture} 
	\FloatBarrier
	
	\caption{Configuration $C_0$ fires, yielding $C_1$. Directed edges depict the flow of chips from richer vertices to poorer vertices.}
	\label{fig:firingexample}

\end{figure}
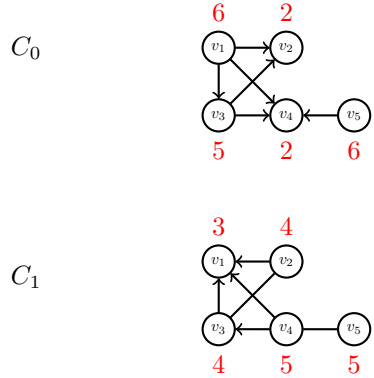

In Parallel Diffusion, the assigned value of a vertex, $v$, at step $t$, is referred to as its \textit{stack size at time $t$}. If the initial configuration is $C$, then the stack size at time $t$ is denoted $|v|_t^C$. This implies that $|v|^C = |v|_0^C$. We omit the superscript when the configuration is clear. Given a graph $G$ and an initial configuration $C_0$, then $C_t=\{(v,|v|^{C}_{t}): v\in V(G)\}$. The \textit{configuration sequence} $Seq(C_0) = \{C_0,C_1,C_2, \dots\}$ is the sequence of configurations that arises as the time increases. The configuration sequence clearly depends on both the initial configuration and the graph $G$. However, it will always be clear to which graph we are referring, so we omit any reference to $G$ in our notation, $Seq(C_0)$. The \textit{initial firing} is the firing of the vertices in $C_0$, yielding $C_1$. Likewise, the \textit{firing at step $m$} is the firing of the vertices of $C_{m}$, yielding $C_{m+1}$.

Given two configurations, $C$ and $D$, of a graph $G$, in which the vertices are labelled, $C$ and $D$ are \textit{equal} if $|v|^C = |v|^D$ for all $v \in V(G)$. Let $Seq(C_0) = \{C_0,C_1,C_2, \dots\}$ be the configuration sequence on a graph $G$ with initial configuration $C_0$. The positive integer $p$ is a \textit{period length} if $C_t = C_{t+p}$ for all $t \geq N$ for some $N$. In this case, $N$ is the \textit{preperiod length}. For such a value, $N$, if $k \geq N$, then we say that the configuration, $C_k$, is \textit{inside} the period. For the purposes of this paper, all references to period length will refer to the \textit{minimum period length $p$} in a given configuration sequence. Also, all references to preperiod length will refer to the \textit{least preperiod length} that yields that minimum period length $p$ in a given configuration sequence. In Figure~\ref{fig:parex}, the period length is 2 and the preperiod length is 3.

In their paper \cite{long}, Long and Narayanan proved the following theorem which was presented as a conjecture by Duffy et al. \cite{duffy}.

\begin{theorem}\cite{long}\label{thm:longnarayanan}
	Every configuration sequence, regardless of initial configuration, has period 1 or 2. 
\end{theorem}

In the proof of Theorem~\ref{thm:longnarayanan}, Long and Narayanan show that once inside the period, if a chip fires from $u$ to $v$ at step $t$, then a chip must fire from $v$ to $u$ at step $t+1$. We will be using this result, so we set it aside as the following corollary.

\begin{corollary}\cite{long} \label{cor:longcor}
	In Parallel Diffusion, let $C_t$ be the configuration at time $t$ and suppose $C_t$ is inside the period. If a vertex $u$ is richer than an adjacent vertex $v$ at step $t$, then $v$ is richer than $u$ at step $t+1$.
\end{corollary}

Note that this implies that if an edge is flat inside the period, then upon firing, it must remain flat. Both the previous corollary and the following observation will prove crucial in counting period configurations on paths.

\begin{observation}\label{lem:inducego}
	A step in Parallel Diffusion induces a graph orientation.
\end{observation}

\begin{proof}
	Let $G$ be a graph and $C_t$ a configuration on $G$. For all pairs of adjacent vertices $u$, $v$ in $G$ at step $t$, either $u$ gives a chip to $v$, $v$ gives a chip to $u$, the stack sizes of $u$ and $v$ are equal in $C_t$. Let $uv$ be an edge. Assign directions as follows:
	
	\begin{itemize}
		
		\item If $u$ gives a chip to $v$ at time $t$, assign $uv$ the edge orientation $u \to v$.
		
		\item If $v$ gives a chip to $u$ at time $t$, assign $uv$ the edge orientation $v \to u$.
		
		\item If the stack sizes of $u$ and $v$ are equal at time $t$, do not direct the edge $uv$.
		
	\end{itemize}
	
	Thus, a graph orientation on $G$ results.
	
\end{proof}

We say that this graph orientation is \textit{induced} by $C_t$, the configuration of $G$ at time $t$. We see an example of a graph orientation induced by a configuration in Parallel Diffusion in Figure~\ref{fig:underlyingorientation}.

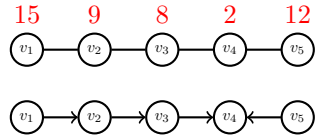
\begin{figure} [H] 
	\centering	
	
	\begin{tikzpicture}[-,-=stealth', auto,node distance=1.5cm,
	thick,scale=0.6, main node/.style={scale=0.6,circle,draw}]
	
	\node[main node, label={[red]90:15}] (1) 					   {$v_1$};						 
	\node[main node, label={[red]90:9}] (2)  [right of=1]        {$v_2$};
	\node[main node, label={[red]90:8}] (3)  [right of=2]        {$v_3$};  
	\node[main node, label={[red]90:2}] (4)  [right of=3]        {$v_4$};						 
	\node[main node, label={[red]90:12}] (5)  [right of=4]        {$v_5$}; 
	\node[main node, label={[red]90:}] (6) 	[below of=1]				   {$v_1$};						 
	\node[main node, label={[red]90:}] (7)  [right of=6]        {$v_2$};
	\node[main node, label={[red]90:}] (8)  [right of=7]        {$v_3$};  
	\node[main node, label={[red]90:}] (9)  [right of=8]        {$v_4$};						 
	\node[main node, label={[red]90:}] (10)  [right of=9]        {$v_5$}; 	
	
	\path (1) edge (2) (2) edge (3) (3) edge (4) (4) edge (5);
	
	\draw [->] (6) edge (7) (7) edge (8) (8) edge (9)
	(10) edge (9) ;

	\end{tikzpicture} 
	\FloatBarrier
	
	\caption{Configuration on $P_5$ and its induced graph orientation.}
	\label{fig:underlyingorientation}

\end{figure}

Let $\dollarsign(G) = \{Seq(C): C$ is a configuration on $G$\}. Since the set of integers is infinite, on any graph $G$, \dollarsign(G) is an infinite set.

Let $\widebar{Seq(C_0)}$ be the singleton or ordered pair of configurations contained within the period of a configuration sequence $Seq(C_0)$. If $Seq(C_0)$ has period 2, define the first element of the ordered pair $\widebar{Seq(C_0)}$ to be the one which occurs first in the configuration sequence. A configuration $D$ on a graph $G$ is a \textit{period configuration} if $D$ is in $\widebar{Seq(C)}$ for some configuration $C$. A configuration $D$ on a graph $G$ is a \textit{$p_2$-configuration} if $D$ is in $\widebar{Seq(C)}$ for some configuration $C$ and $\widebar{Seq(C)}$ has 2 elements. A configuration $D$ on a graph $G$ is a \textit{fixed configuration} if $D$ is in $\widebar{Seq(C)}$ for some configuration $C$ and $\widebar{Seq(C)}$ has exactly 1 element. A \textit{period orientation} is a graph orientation that is induced by a period configuration. A \textit{$p_2$-orientation} is a graph orientation that is induced by a $p_2$-configuration. A \textit{fixed orientation} is a graph orientation that is induced by a fixed configuration.

Let $C$ be a configuration on a graph $G$. Let $C+k$ be the configuration created by adding an integer $k$ to every stack size in the configuration $C$. $Seq(C), Seq(D) \in \dollarsign(G)$ are \textit{isomorphic} if $\widebar{Seq(C+k)} = \widebar{Seq(D)}$ for some integer $k$. Let $\widebar{\dollarsign(G)} = \{\widebar{Seq(C)}: C$ a configuration on $G\}$. Let $\widebar{\dollarsign}'(G)$ be the largest subset of $\widebar{\dollarsign}(G)$ such that no two elements are isomorphic. In this paper, we will determine the cardinality of $\widebar{\dollarsign}'(P_n)$ for all $n \geq 1$.

We see an example of isomorphic configuration sequences in Figure~\ref{fig:isomorphic}. We see an example of $\widebar{\dollarsign}(G)$ and $\widebar{\dollarsign}'(G)$ in Figure~\ref{fig:L'}.

\begin{figure}[H]
	\centering
	\begin{tikzpicture}[-,-=stealth', auto,node distance=1.5cm,
	thick,scale=0.8, main node/.style={scale=0.8,circle,draw}]
	\node[main node] (1) {};
	\node[main node] [right of=1] (2) {};
	
	\draw [->] (1) edge (2);
	\end{tikzpicture}
	
	\small{$\widebar{\dollarsign}(P_2) = \{\{(0,0)\},\{(0,1),(1,0)\},\{(1,0),(0,1)\}, \{(1,2),(2,1)\}, \{(2,1),(1,2)\}, \{(3,2),(2,3)\}, \dots \}$}
	
	\vspace{0.5cm}
	
	$\widebar{\dollarsign}'(P_2) = \{\{(0,0)\}, \{(0,1),(1,0)\}, \{(1,0),(0,1)\} \}$	
	
	\caption{$\widebar{\dollarsign}(P_2)$ and $\widebar{\dollarsign}'(P_2)$}
	\label{fig:L'}	
	
\end{figure}

\begin{figure}[H]
	
	Configuration sequence $Seq(C_0)$
	
	\[
	\begin{tikzpicture}[-,-=stealth', auto,node distance=1.5cm,
	thick,scale=0.6, main node/.style={scale=0.6,circle,draw}]
	
	\node[main node, label={[red]90:0}] (1) 					   {$v_1$};						 
	\node[main node, label={[red]90:1}] (2)  [right of=1]        {$v_2$};
	\node[main node, label={[red]90:1}] (3)  [right of=2]        {$v_3$};  
	\node[main node, label={[red]90:1}] (4)  [right of=3]        {$v_4$};						 
	\node[main node, label={[red]90:0}] (5)  [right of=4]        {$v_5$}; 
	\node[draw=none, fill=none]         (0)  [left of=1]         {$C_0$};
	\path
	
	(2) edge (3)
	(3) edge (4); 
	
	\draw [->] (2) edge (1) 
	(4) edge (5) ;

	\end{tikzpicture} \]

	\FloatBarrier

	Firing at step 0

	\[
	\begin{tikzpicture}[-,-=stealth', auto,node distance=1.5cm,
	thick,scale=0.6, main node/.style={scale=0.6,circle,draw}]
	
	\node[main node, label={[red]90:1}] (1) 					   {$v_1$};						 
	\node[main node, label={[red]90:0}] (2)  [right of=1]        {$v_2$};
	\node[main node, label={[red]90:1}] (3)  [right of=2]        {$v_3$};  
	\node[main node, label={[red]90:0}] (4)  [right of=3]        {$v_4$};						 
	\node[main node, label={[red]90:1}] (5)  [right of=4]        {$v_5$}; 
	\node[draw=none, fill=none]         (0)  [left of=1]         {$C_1$};
	\draw [->] (1) edge (2) (3) edge (2) (3) edge (4)
	(5) edge (4) ;

	\end{tikzpicture} \]
	
	\FloatBarrier

	Firing at step 1
	
	\[
	\begin{tikzpicture}[-,-=stealth', auto,node distance=1.5cm,
	thick,scale=0.6, main node/.style={scale=0.6,circle,draw}]
	
	\node[main node, label={[red]90:0}] (1) 					   {$v_1$};						 
	\node[main node, label={[red]90:2}] (2)  [right of=1]        {$v_2$};
	\node[main node, label={[red]90:-1}] (3)  [right of=2]        {$v_3$};  
	\node[main node, label={[red]90:2}] (4)  [right of=3]        {$v_4$};						 
	\node[main node, label={[red]90:0}] (5)  [right of=4]        {$v_5$}; 
	\node[draw=none, fill=none]         (0)  [left of=1]         {$C_2$};
	\draw [->] (2) edge (1) (2) edge (3) (4) edge (3)
	(4) edge (5) ;

	\end{tikzpicture} \]
	
	\FloatBarrier
	
	Firing at step 2
	
	\[
	\begin{tikzpicture}[-,-=stealth', auto,node distance=1.5cm,
	thick,scale=0.6, main node/.style={scale=0.6,circle,draw}]
	
	\node[main node, label={[red]90:1}] (1) 					   {$v_1$};						 
	\node[main node, label={[red]90:0}] (2)  [right of=1]        {$v_2$};
	\node[main node, label={[red]90:1}] (3)  [right of=2]        {$v_3$};  
	\node[main node, label={[red]90:0}] (4)  [right of=3]        {$v_4$};						 
	\node[main node, label={[red]90:1}] (5)  [right of=4]        {$v_5$}; 
	\node[draw=none, fill=none]         (0)  [left of=1]         {$C_3$};
	\draw [->] (1) edge (2) (3) edge (2) (3) edge (4)
	(5) edge (4) ;

	\end{tikzpicture} \]
	
	\FloatBarrier

	\vspace{1cm}

	Configuration sequence $Seq(C'_0)$
	
	\[
	\begin{tikzpicture}[-,-=stealth', auto,node distance=1.5cm,
	thick,scale=0.6, main node/.style={scale=0.6,circle,draw}]
	
	\node[main node, label={[red]90:-1}] (1) 					   {$v_1$};						 
	\node[main node, label={[red]90:1}] (2)  [right of=1]        {$v_2$};
	\node[main node, label={[red]90:-2}] (3)  [right of=2]        {$v_3$};  
	\node[main node, label={[red]90:1}] (4)  [right of=3]        {$v_4$};						 
	\node[main node, label={[red]90:-1}] (5)  [right of=4]        {$v_5$}; 
	\node[draw=none, fill=none]         (0)  [left of=1]         {$C'_0$};
	\draw [->] (1) edge (2) (3) edge (2) (3) edge (4)
	(5) edge (4) ;

	\end{tikzpicture} \]
	
	\FloatBarrier
	
	Firing at step 0
	
	\[
	\begin{tikzpicture}[-,-=stealth', auto,node distance=1.5cm,
	thick,scale=0.6, main node/.style={scale=0.6,circle,draw}]
	
	\node[main node, label={[red]90:0}] (1) 					   {$v_1$};						 
	\node[main node, label={[red]90:-1}] (2)  [right of=1]        {$v_2$};
	\node[main node, label={[red]90:0}] (3)  [right of=2]        {$v_3$};  
	\node[main node, label={[red]90:-1}] (4)  [right of=3]        {$v_4$};						 
	\node[main node, label={[red]90:0}] (5)  [right of=4]        {$v_5$}; 
	\node[draw=none, fill=none]         (0)  [left of=1]         {$C'_1$};
	
	\draw [->] (2) edge (1) (2) edge (3) (4) edge (3) 
	(4) edge (5) ;

	\end{tikzpicture} \]
	
	\FloatBarrier
	
	Firing at step 1
	
	\[
	\begin{tikzpicture}[-,-=stealth', auto,node distance=1.5cm,
	thick,scale=0.6, main node/.style={scale=0.6,circle,draw}]
	
	\node[main node, label={[red]90:-1}] (1) 					   {$v_1$};						 
	\node[main node, label={[red]90:1}] (2)  [right of=1]        {$v_2$};
	\node[main node, label={[red]90:-2}] (3)  [right of=2]        {$v_3$};  
	\node[main node, label={[red]90:1}] (4)  [right of=3]        {$v_4$};						 
	\node[main node, label={[red]90:-1}] (5)  [right of=4]        {$v_5$}; 
	\node[draw=none, fill=none]         (0)  [left of=1]         {$C'_2$};
	\draw [->] (1) edge (2) (3) edge (2) (3) edge (4)
	(5) edge (4);

	\end{tikzpicture} \]
	
	\FloatBarrier
	
	Firing at step 2
	
	\[
	\begin{tikzpicture}[-,-=stealth', auto,node distance=1.5cm,
	thick,scale=0.6, main node/.style={scale=0.6,circle,draw}]
	
	\node[main node, label={[red]90:0}] (1) 					   {$v_1$};						 
	\node[main node, label={[red]90:-1}] (2)  [right of=1]        {$v_2$};
	\node[main node, label={[red]90:0}] (3)  [right of=2]        {$v_3$};  
	\node[main node, label={[red]90:-1}] (4)  [right of=3]        {$v_4$};						 
	\node[main node, label={[red]90:0}] (5)  [right of=4]        {$v_5$}; 
	\node[draw=none, fill=none]         (0)  [left of=1]         {$C'_3$};
	\draw [->] (2) edge (1) (2) edge (3) (4) edge (3)
	(4) edge (5) ;

	\end{tikzpicture} \]
	
	\FloatBarrier
	
	\caption{Two isomorphic configuration sequences.}
	
	\label{fig:isomorphic}
	
\end{figure}

\begin{lemma}\label{lem:onlyone} (Lemma 3.1.16 from ``On Variants of Diffusion", T. Mullen, PhD Thesis)
	Let $G$ be a graph. Up to isomorphism, the only fixed configuration on $G$ is the one in which every vertex has $0$ chips. 
\end{lemma}

\begin{lemma} (Lemma 3.1.1 from ``On Variants of Diffusion", T. Mullen, PhD Thesis) \label{lem:k-scale}
	Let $C$ and $D$ be configurations on a graph $G$. Let $k$ be an integer. Suppose that for all $v \in V(G)$, $|v|^{C} = |v|^{D} + k$. Then for all $t$, $|v|^{C}_t = |v|^{D}_t + k$. 
\end{lemma}

\section{Paths}
\label{sec:1}

We now approach the problem of counting all of the $p_2$-orientations on a path. We will draw our paths along a horizontal axis and label the vertices from \textit{right to left} with the rightmost vertex labelled $v_1$. We will fix $v_1$ at zero chips. By Theorem~\ref{thm:longnarayanan}, we know every configuration is eventually periodic with either period 1 or 2. Since we are fixing $v_1$ at zero chips, by Lemma~\ref{lem:onlyone}, the only possible period 1 configuration on any path is the one in which every vertex has zero chips. So, we will restrict our view to only counting $p_2$-configurations. We will make frequent use of Corollary~\ref{cor:longcor} in justifying the flow of chips inside the period.

Since our paths are drawn along a horizontal axis, the terms ``left" and ``right" have an obvious meaning. On a path drawn along a horizontal axis, a \textit{left} edge is a directed edge in which the head is to the left of the tail and a \textit{right} edge is a directed edge in which the head is to the right of the tail. When referring to edges contained within a path defined on a horizontal axis, two edges \textit{agree} if they are either both right edges or both left edges. Two edges \textit{disagree} if one is right and the other is left. An \textit{alternating arrow orientation} is a path orientation in which every pair of adjacent edges disagree. Note that this means an alternating arrow orientation cannot contain any flat edges.

We will proceed by dividing $P_n$, $n \geq 3$, into all of its possible $p_2$-orientations and then determining, for each $p_2$-orientation, $R$, how many different nonisomorphic $p_2$-configurations induce $R$. We begin by characterizing when a path orientation is a $p_2$-orientation.

\begin{theorem} \label{thm:illegalorientations}A path orientation is a $p_2$-orientation if and only if none of the following mixed graphs exist as a suborientation (let a square represent a leaf and a circle represent a vertex that may or may not be a leaf).
	
	
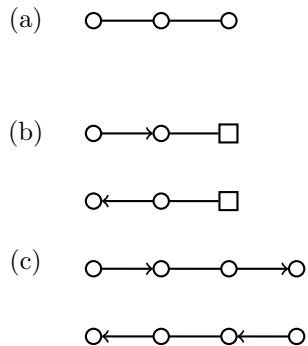
\[
	\begin{tikzpicture}[-,-=stealth', auto,node distance=1.5cm,
	thick,scale=0.6, main node/.style={scale=0.6,circle,draw}]
	\node[draw=none, fill=none] (50)          			{\emph{(a)}};
	\node[main node] (-2) 	[right of=50]   			    {};                  
	\node[main node] (-1) 	[right of=-2]  			     {};                 
	\node[main node] (0)  	[right of=-1]   		    {};
	\node[draw=none, fill=none] (51) [below of=50]         {\emph{(b)}};                   
	\node[main node] (1) 		[right of=51]	    		{};						 
	\node[main node] (2)  	  [right of=1]      		  {};
	\node[rectangle,draw] (3)   [right=0.65cm of 2]        {}; 
	\node[main node] (4) 		[below of=1]        		{};						 
	\node[main node] (5)  	[right of=4]       			 {}; 
	\node[rectangle,draw] (6)	  [right=0.65cm of 5]        {}; 
	\node[draw=none, fill=none] (52) [below=1.1cm of 51] {\emph{(c)}};
	\node[main node] (7)  	[below of=4]        		{};
	\node[main node] (8)  	[right of=7]       			 {};
	\node[main node] (9)  	[right of=8]      			  {};
	\node[main node] (10) 	 [below of=7]     			   {};
	\node[main node] (11) 	 [right of=10]    			    {};
	\node[main node] (12) 	 [right of=11]      			  {};
	\node[main node] (21) 	 [right of=9]       			  {};
	\node[main node] (22) 	 [right of=12]      			  {};
	
	\path		 
	(-2) edge node [] {} (-1)
	(-1) edge node [] {} (0)
	(2) edge node [] {} (3)			
	(5) edge node [] {} (6)
	(8) edge node [] {} (9)
	(11) edge node [] {} (12);
	
	\draw [->] (1) edge (2) (5) edge (4) (7) edge (8)
	(11) edge (10) (9) edge (21)
	(22) edge (12);

	\end{tikzpicture} \]
	
	\FloatBarrier
	
	In addition,
	
	
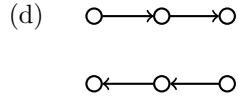
\[
	\begin{tikzpicture}[-,-=stealth', auto,node distance=1.5cm,
	thick,scale=0.6, main node/.style={scale=0.6,circle,draw}]                               
	\node[draw=none, fill=none] (0)               {\emph{(d)}};
	\node[main node] (1) 	[right of=0]				    {};						 
	\node[main node] (2)  [right of=1]        {};
	\node[main node] (3)  [right=0.65cm of 2]        {};  
	\node[main node] (4) 	[below of=1]        {};						 
	\node[main node] (5)  [right of=4]        {}; 
	\node[main node] (6)  [right=0.65cm of 5]        {};

	\draw [->] (1) edge (2) (2) edge (3) (5) edge (4)
	(6) edge (5);

	\end{tikzpicture} \]
	
	\FloatBarrier
	can only exist within the respective subgraphs
	
	
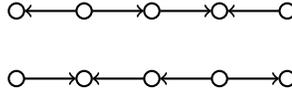
\[
	\begin{tikzpicture}[-,-=stealth', auto,node distance=1.5cm,
	thick,scale=0.6, main node/.style={scale=0.6,circle,draw}]                               
	\node[main node] (1) 					    {};						 
	\node[main node] (2)  [right of=1]        {};
	\node[main node] (3)  [right of=2]        {};  
	\node[main node] (4) 	[below of=1]        {};						 
	\node[main node] (5)  [right of=4]        {}; 
	\node[main node] (6)  [right of=5]        {}; 
	\node[main node] (7)  [right of=3]        {};  
	\node[main node] (8) 	[right of=7]        {};						 
	\node[main node] (9)  [right of=6]        {}; 
	\node[main node] (10)  [right of=9]        {};

	\draw [->] (2) edge (1) (2) edge (3) (4) edge (5)
	(6) edge (5) (3) edge (7) (8) edge (7)
	(9) edge (6) (9) edge (10);

	\end{tikzpicture} \]
	\FloatBarrier
	
\end{theorem}

We express the proof with a series of lemmas. In particular, we prove the necessary condition with Lemmas~\ref{lem:noflat}-\ref{lem:agreeingarrowsbookend}, and we prove the sufficient statement with Lemma~\ref{lem:iff}. 

\begin{lemma}\label{lem:noflat}
	Case (a) in Theorem~\ref{thm:illegalorientations}: No $p_2$-configuration on a path induces a mixed graph with two adjacent flat edges.
\end{lemma}

\begin{proof}
	Suppose, by contradiction, that it were possible to have two flat edges adjacent to each other in a graph orientation induced by a $p_2$-configuration. 
	Call these edges $e_k$ and $e_{k+1}$, and call the vertices $v_k$, $v_{k+1}$, and $v_{k+2}$. So, the subpath in question is $X=v_{k+2}e_{k+1}v_{k+1}e_kv_k$.
	Since the period is two, we know that some edge in the graph must not be flat by Lemma~\ref{lem:onlyone}.
	Without loss of generality, suppose that the edge immediately to the right of $X$, $e_{k-1}$, is oriented left or right. This would result in $v_{k+1}$ maintaining its number of chips in the initial firing, but as for the firing at step $1$, $v_{k+1}$ is now adjacent to a vertex that either increased or decreased its stack size in the initial firing. Thus, $|v_{k+1}|_2 \neq |v_{k+1}|_0$. Therefore, the orientation is not induced by a $p_2$-configuration. This is a contradiction.
\end{proof}

\begin{lemma}\label{lem:endpoint}
	Case (b) in Theorem~\ref{thm:illegalorientations}: No $p_2$-configuration on a path induces a graph orientation which contains a flat edge incident with a leaf.
\end{lemma}

\begin{proof}
	Suppose, by contradiction, that there exists a flat edge, $e_1$, incident with a leaf, $v_1$, in a graph orientation, $R$, induced by a $p_2$-configuration, $C$. We know by Lemma~\ref{lem:onlyone} that there exists at least one more edge, $e_2$, in this graph since we are supposing that $C$ has period 2. In step $1$, $v_1$ and its neighbour, $v_2$, have equal stack sizes. By Lemma~\ref{lem:noflat}, we know that $e_2$, the edge adjacent to $e_1$, is not flat. So, $v_2$ has either gained or lost a chip in the initial firing, while $|v_1|$ has gone unchanged. So, in the firing at step $1$, $v_1$ will gain or lose a chip, indicating that $|v_1|_2 \neq |v_1|_0$. This is a contradiction.
\end{proof}

\begin{lemma} \label{lem:rfl}
	Case (c) in Theorem~\ref{thm:illegalorientations}: Every flat edge in a graph orientation induced by a $p_2$-configuration on a path is adjacent to two edges: one right and one left.
\end{lemma}

\begin{proof}
	Suppose, by contradiction, that there exists some flat edge that is not adjacent to both a left edge and a right edge in some graph orientation, $R$, induced by a $p_2$-configuration, $C$. Call this flat edge $e_k$ and its endpoints $v_k$ and $v_{k+1}$. In the initial firing, no chips will move across $e_k$ since it is flat. In the firing at step 1, the same must be true since we supposed that $C$ is a $p_2$-configuration.
	
	\textbf{Case 1:} $e_k$ is incident with a leaf.
	By Lemma~\ref{lem:endpoint}, we know that no flat can be incident with a leaf. This is a contradiction.
	
	\textbf{Case 2:} At least one edge adjacent to $e_k$ is flat. By Lemma~\ref{lem:noflat}, this is impossible. This is a contradiction.
	
	\textbf{Case 3:} $e_k$ is adjacent to two left edges.
	Then $|v_k|_1 = |v_k|_0 + 1$ and $|v_{k+1}|_1 = |v_k|_0 - 1$. So, the stack sizes of $v_k$ and $v_{k+1}$ are not equal at step $t+1$. This is a contradiction.
	
	\textbf{Case 4:} $e_k$ is adjacent to two right edges.
	Then $|v_k|_1 = |v_k|_0 - 1$ and $|v_{k+1}|_1 = |v_k|_0 + 1$. So, the stack sizes of $v_k$ and $v_{k+1}$ are not equal at step $t+1$. This is a contradiction.
	
	Thus, every flat edge in a $p_2$-configuration on a path is incident with two edges: one right and one left.
\end{proof}

\begin{lemma}\label{lem:agreeingarrowsbookend}
	Case (d) in Theorem~\ref{thm:illegalorientations}: Let $H$ be a suborientation of a $p_2$-orientation, $R$, on a path, $P_n$. If $H$ is a directed path consisting of three vertices and two right edges, then $H$ must be incident with two left edges in $P_n$. Conversely, if $H$ is a directed path consisting of three vertices and two left edges, then $H$ must be incident with two right edges in $P_n$.
\end{lemma}

\begin{proof}
	Suppose that $H=v_{k+1}e_kv_ke_{k-1}v_{k-1}$ contains three vertices and, without loss of generality, two right edges. In the initial firing, $v_k$ gives and receives a single chip, maintaining its stack size. Initially, we have $|v_{k+1}|_0 > |v_k|_0 > |v_{k-1}|_0$. In the configuration at step 1, these inequalities must be reversed since we are already inside the period, by Corollary~\ref{cor:longcor}. Since $|v_k|_1 = |v_k|_0$, we know that $v_{k+1}$ must lose at least 2 chips in the initial firing and $v_{k-1}$ must gain at least 2 chips in the initial firing. However, this is only possible if both edges in $P_n-H$ that are incident with $H$ are left edges.

\end{proof}

\begin{lemma}\label{lem:iff}
	Any path orientation with no suborientations of the forms outlined in Theorem~\ref{thm:illegalorientations}, is a $p_2$-orientation.
\end{lemma}

\begin{proof}
	We must now show that every orientation that does not contain any of the suborientations from Theorem~\ref{thm:illegalorientations} is a $p_2$-orientation. Our method will involve taking an arbitrary orientation $R$ that does not contain any of the suborientations listed in Theorem~\ref{thm:illegalorientations}, and proving that there exists an assignment of stack sizes that both induces $R$ and exists within a period of length 2.
	There are 3 orientations that an edge may have: flat, left, and right. We will assume that moving from right to left, every vertex $v_i$ has been assigned an initial stack size to create the configuration, $C$, using the following rule:
	
	$$
	|v_i|_0 = 
	\begin{cases}
	|v_{i-1}|_0 + 1$ if edge $v_{i-1}v_i$ is directed right. $\\
	|v_{i-1}|_0 - 1$ if edge $v_{i-1}v_i$ is directed left. $\\
	|v_{i-1}|_0$ if edge $v_{i-1}v_i$ is flat. $\\
	\end{cases}
	$$
	
	and $v_1$ has been assigned 0 chips.
	
	We now inspect an edge $e_j = v_jv_{j+1}$ in $R$ with the goal of determining if its incident vertices will restore their initial stack size after two firings.
	
	\textbf{Case 1:} The edge $e_j = v_jv_{j+1}$ is flat.
	
	We know that neither $v_j$ nor $v_{j+1}$ is a leaf by Lemma~\ref{lem:endpoint}. In $C$, $v_j$ and $v_{j+1}$ have both been initially assigned to have the same number of chips. However, in order for $C$ to be a $p_2$-configuration, we must also have that $|v_j|_1 = |v_{j+1}|_1$, by Corollary~\ref{cor:longcor}. In order to determine this, we must know the stack sizes of $v_j$ and $v_{j+1}$ at step 1. This will depend on the initial orientation of edges $e_{j-1} = v_{j-1}v_j$ and $e_{j+1} = v_{j+1}v_{j+2}$. We know that since $e_j$ is flat, no adjacent edge can be flat by Lemma~\ref{lem:noflat}. Also, $e_{j-1}$ and $e_{j+1}$ cannot be both right or both left by Lemma~\ref{lem:rfl}. So, $e_{j-1}$ and $e_{j+1}$ must disagree. Without loss of generality, suppose $e_{j-1}$ is directed right and $e_{j+1}$ is directed left. So, our rule dictates that $|v_{j-1}|_0 + 1 = |v_{j}|_0 = |v_{j+1}|_0 = |v_{j+2}|_0 + 1$. Thus, with both vertices receiving a total of one chip at step 0, we have that $|v_{j}|_1 = |v_{j+1}|_1$. So, $|v_j|_2 = |v_j|_0$ and $|v_{j+1}|_2 = |v_j|_0$.

	\textbf{Case 2:} The edge $e_j = v_jv_{j+1}$ is directed.
	
	Suppose, without loss of generality, that $e_j$ is directed right. In $C$, $|v_{j+1}|_0 = |v_j|_0 + 1$. In order for $C$ to be a $p_2$-configuration, we must have that $|v_{j+1}|_1 < |v_j|_1$. In order to determine this, we must know the stack sizes of $v_j$ and $v_{j+1}$ at step 1. This will depend on the initial orientation of edges $e_{j-1} = v_{j-1}v_j$ and $e_{j+1} = v_{j+1}v_{j+2}$. Note that either $e_{j-1}$ or  $e_{j+1}$ may not exist depending on if either $v_j$ or $v_{j+1}$ is a leaf. However, the absence of either of these edges has the same effect on the stack size of the incident vertices as a flat edge would. We consider the possible orientations of $e_{j-1}$ and $e_{j+1}$.
	
	\begin{enumerate}[(i)]
		
		\item \textit{Both $e_{j-1}$ and $e_{j+1}$ are flat.} 
		
		So, $|v_{j+1}|_1 = |v_{j+1}|_0 - 1$ and $|v_j|_1 = |v_j|_0 + 1$. Thus, $|v_{j+1}|_1 = |v_j|_0 < |v_j|_0 + 1 = |v_j|_1$. 
		
		\item \textit{$e_{j-1}$ is flat and $e_{j+1}$ is directed left.}
		
		So, $|v_{j+1}|_1 = |v_{j+1}|_0 - 2$ and $|v_j|_1 = |v_j|_0 + 1$. Thus, $|v_{j+1}|_1 = |v_j|_0 - 1 < |v_j|_0 + 1 = |v_j|_1$.
		
		\item \textit{$e_{j-1}$ is flat and $e_{j+1}$ is directed right.}
		
		This suborientation cannot exist within the period by Lemma~\ref{lem:agreeingarrowsbookend}.
		
		\item \textit{$e_{j-1}$ is directed right and $e_{j+1}$ is flat.}
		
		This suborientation cannot exist within the period by Lemma~\ref{lem:agreeingarrowsbookend}.
		
		\item \textit{$e_{j-1}$ is directed right and $e_{j+1}$ is directed left.}
		
		So, $|v_{j+1}|_1 = |v_{j+1}|_0 - 2$ and $|v_j|_1 = |v_j|_0$. Thus, $|v_{j+1}|_1 = |v_j|_0 - 1 < |v_j|_0 = |v_j|_1$.
		
		\item \textit{Both $e_{j-1}$ and $e_{j+1}$ are directed right.}
		
		This suborientation cannot exist within the period by Lemma~\ref{lem:agreeingarrowsbookend}.
		
		\item \textit{$e_{j-1}$ is directed left and $e_{j+1}$ is flat.}
		
		So, $|v_{j+1}|_1 = |v_{j+1}|_0 - 1$ and $|v_j|_1 = |v_j|_0 + 2$. Thus, $|v_{j+1}|_1 = |v_j|_0 < |v_j|_0 + 2 = |v_j|_1$.
		
		\item \textit{Both $e_{j-1}$ and $e_{j+1}$ are directed left.}
		
		So, $|v_{j+1}|_1 = |v_{j+1}|_0 - 2$ and $|v_j|_1 = |v_j|_0 + 2$. Thus, $|v_{j+1}|_1 = |v_j|_0 - 1 < |v_j|_0 + 2 = |v_j|_1$.
		
		\item \textit{$e_{j-1}$ is directed left and $e_{j+1}$ is directed right.}
		
		So, $|v_{j+1}|_1 = |v_{j+1}|_0$ and $|v_j|_1 = |v_j|_0 + 2$. Thus, $|v_{j+1}|_1 = |v_j|_0 + 1 < |v_j|_0 + 2 = |v_j|_1$.
		
	\end{enumerate}
	
	So, for all possible graph orientations $R$, either $R$ is a $p_2$-orientation or $R$ contains a suborientation listed in Theorem~\ref{thm:illegalorientations}.

\end{proof}

Let $R_n$ be the number of $p_2$-orientations on $P_n$. Quick calculations show that $R_1 = 0$, $R_2 = 2$, $R_3 = 2$, and $R_4 = 4$.

\begin{theorem}\label{thm:countorientations}
	The number of $p_2$-orientations, $R_n$, on a path $P_n$, $n \geq 5$, is given by the recurrence relation $R_n = R_{n-1} + 2R_{n-2} - R_{n-4}$ with initial values $R_1 = 0$, $R_2 = 2$, $R_3 = 2$, and $R_4 = 4$.
\end{theorem}

\begin{proof}
	
	Let $R$ be a $p_2$-orientation on $P_n=v_1e_1v_2e_2v_3 \dots v_{n-1}e_{n-1}v_n$. There are three mutually exclusive and exhaustive cases: $e_{n-2}$ is flat, $e_{n-2}$ agrees with $e_{n-3}$, or $e_{n-2}$ is neither flat nor agreeing with $e_{n-3}$. We will add together the total number of $p_2$-configurations of each form to reach $R_n = R_{n-1} + 2R_{n-2} - R_{n-4}$.
	
	\textbf{Case 1:} $e_{n-2}$ is flat.
	Let $R'$ be the induced suborientation of $R$ on\\ $P_{n-2}=v_1e_1v_2 \dots v_{n-3}e_{n-3}v_{n-2}$. We now check that $R'$ is a $p_2$-orientation by using our criteria from Lemmas~\ref{lem:noflat} - \ref{lem:agreeingarrowsbookend}.

	Lemma~\ref{lem:noflat} states the non-existence of adjacent flat edges. Since $R$ is a $p_2$-orientation, it does not contain adjacent flat edges. Therefore $R'$, being an induced suborientation of $R$, also does not contain adjacent flat edges.
	
	Lemma~\ref{lem:endpoint} states the non-existence of flat edges incident with a leaf. Since $R$ is a $p_2$-orientation, it does not contain a flat edge incident with a leaf. The vertex $v_{n-2}$ is a leaf in $R'$ but not in $R$. However, we know, by Lemma~\ref{lem:rfl}, that $e_{n-3}$ and $e_{n-1}$ disagree. This implies that $e_{n-3}$, the only edge incident with $v_{n-2}$ in $R'$, is not flat. Therefore, $R'$ does not contain a flat edge incident with a leaf. 
	
	Lemma~\ref{lem:rfl} states that flat edges must be adjacent to disagreeing edges. Since $R$ is a $p_2$-orientation, each flat edge in $R$ is adjacent to disagreeing edges. Since $e_{n-3}$ is not flat, every flat edge in $R'$ is adjacent to the same set of edges in both $R$ and $R'$. Therefore, every flat edge in $R'$ is adjacent to disagreeing edges.
	
	Lemma~\ref{lem:agreeingarrowsbookend} states that any edge, $e$, adjacent to an edge ,$f$, with which it agrees must also be adjacent to an edge, $d$, with which it disagrees. We know by Lemma~\ref{lem:agreeingarrowsbookend}, that since $R$ is a $p_2$-orientation and $e_{n-2}$ is flat, $e_{n-3}$ does not agree with $e_{n-4}$. So, every pair of adjacent agreeing edges in $R'$ is adjacent to the same set of edges in both $R$ and $R'$. Therefore, in $R'$, every edge, $e$, adjacent to an edge, $f$, with which it agrees is also adjacent to an edge, $d$, with which it disagrees.
	
	By Theorem~\ref{thm:illegalorientations}, we can conclude that $R'$ is a $p_2$-orientation. Also, $R$ is uniquely determined by $R'$. That is, given $R'$ and that $e_{n-2}$ is flat, we know that $R$ must have an $e_{n-1}$ that disagrees with $e_{n-3}$. Therefore, there exist exactly $R_{n-2}$ different $p_2$-orientations of this form on $P_n$.

	\textbf{Case 2:} $e_{n-2}$ agrees with $e_{n-3}$.
	
	Let $R'$ be the induced suborientation of $R$ on $P_{n-2}=v_1e_1v_2 \dots v_{n-3}e_{n-3}v_{n-2}$. We must check that $R'$ is a $p_2$-orientation by using our criteria from Lemmas~\ref{lem:noflat} - \ref{lem:agreeingarrowsbookend}. 
	
	Lemma~\ref{lem:noflat} states the non-existence of adjacent flat edges. Since $R$ is a $p_2$-orientation, it does not contain adjacent flat edges. Therefore $R'$, being an induced suborientation of $R$, also does not contain adjacent flat edges.
	
	Lemma~\ref{lem:endpoint} states the non-existence of flat edges incident with a leaf. Since $R$ is a $p_2$-orientation, it does not contain a flat edge incident with a leaf. The vertex $v_{n-2}$ is a leaf in $R'$ but not in $R$. However, we know that $e_{n-3}$ and $e_{n-2}$ agree. This implies that $e_{n-3}$, the only edge incident with $v_{n-2}$ in $R'$, is not flat. Therefore, $R'$ does not contain a flat edge incident with a leaf. 
	
	Lemma~\ref{lem:rfl} states that flat edges must be adjacent to disagreeing edges. Since $R$ is a $p_2$-orientation, each flat edge in $R$ is adjacent to disagreeing edges. Since $e_{n-3}$ is not flat, every flat edge in $R'$ is adjacent to the same set of edges in both $R$ and $R'$. Therefore, every flat edge in $R'$ is adjacent to disagreeing edges.
	
	Lemma~\ref{lem:agreeingarrowsbookend} states that any edge, $e$, adjacent to an edge, $f$, with which it agrees must also be adjacent to an edge, $d$, with which it disagrees. We know by Lemma~\ref{lem:agreeingarrowsbookend}, that since $R$ is a $p_2$-orientation and $e_{n-2}$ agrees with $e_{n-3}$, then $e_{n-3}$ disagrees with $e_{n-4}$. So, every pair of adjacent agreeing edges in $R'$ is adjacent to the same set of edges in both $R$ and $R'$. Therefore, in $R'$, every edge, $e$, adjacent to an edge, $f$, with which it agrees is also adjacent to an edge, $d$, with which it disagrees.
	
	By Theorem~\ref{thm:illegalorientations}, we can conclude that $R'$ is a $p_2$-orientation. Also, $R$ is uniquely determined by $R'$. That is, given $R'$ and that $e_{n-2}$ agrees with $e_{n-3}$, we know that $R$ must have an $e_{n-1}$ that disagrees with $e_{n-2}$. So, the number of $p_2$-orientations of $P_n$ in which $e_{n-2}$ agrees with $e_{n-3}$ is equal to the number of $p_2$-orientations of $R'$ in which $e_{n-3}$ disagrees with $e_{n-4}$. We can determine this value recursively. From Case 1, we can see that the number of $p_2$-orientations of $R'$ in which $e_{n-3}$ disagrees with $e_{n-4}$ is equal to $R_{n-2} - R_{n-4}$.

	\textbf{Case 3:} $e_{n-2}$ is neither flat nor agreeing with $e_{n-3}$.
	
	Let $R'$ be the induced suborientation of $R$ on $P_{n-1}=v_1e_1v_2 \dots v_{n-3}e_{n-3}v_{n-2}e_{n-2}v_{n-1}$. We must check that $R'$ is a $p_2$-orientation by using our criteria from Lemmas~\ref{lem:noflat} - \ref{lem:agreeingarrowsbookend}. 
	
	Lemma~\ref{lem:noflat} states the non-existence of adjacent flat edges. Since $R$ is a $p_2$-orientation, it does not contain adjacent flat edges. Therefore $R'$, being an induced suborientation of $R$, also does not contain adjacent flat edges.
	
	Lemma~\ref{lem:endpoint} states the non-existence of flat edges incident with a leaf. Since $R$ is a $p_2$-orientation, it does not contain a flat edge incident with a leaf. The vertex $v_{n-1}$ is a leaf in $R'$ but not in $R$. However, we know that $e_{n-2}$, the only edge incident with $v_{n-1}$ in $R'$, is not flat. Therefore, $R'$ does not contain a flat edge incident with a leaf. 
	
	Lemma~\ref{lem:rfl} states that flat edges must be adjacent to disagreeing edges. Since $R$ is a $p_2$-orientation, each flat edge in $R$ is adjacent to disagreeing edges. Since $e_{n-2}$ is not flat, every flat edge in $R'$ is adjacent to the same set of edges in both $R$ and $R'$. Therefore, every flat edge in $R'$ is adjacent to disagreeing edges.
	
	Lemma~\ref{lem:agreeingarrowsbookend} states that any edge, $e$, adjacent to an edge, $f$, with which it agrees must also be adjacent to an edge, $d$, with which it disagrees. Since $e_{n-2}$ does not agree with $e_{n-3}$, every pair of adjacent agreeing edges in $R'$ is adjacent to the same set of edges in both $R$ and $R'$. Therefore, in $R'$, every edge, $e$, adjacent to an edge, $f$, with which it agrees is also adjacent to an edge, $d$, with which it disagrees.
	
	By Theorem~\ref{thm:illegalorientations}, we can conclude that $R'$ is a $p_2$-orientation. Also, $R$ is uniquely determined by $R'$. That is, given $R'$, we know that $R$ must have an $e_{n-1}$ which disagrees with $e_{n-2}$. Therefore, there exist exactly $R_{n-1}$ different $p_2$-orientations of this form on $P_n$. 
	
	Adding together the values from our three cases, we get that $R_n = R_{n-1} + 2R_{n-2} - R_{n-4}$.
	
\end{proof}

In the OEIS \cite{oeis}, this sequence: $0,2,2,4,8,14,28,52,100,190,362...$ generated by the recurrence in Theorem~\ref{thm:countorientations}, is A052535. The generating sequence is $\frac{(1 - x^2)}{(1 - x - 2x^2 + x^4)}$. The asymptotic solution for the $k^{th}$ term of this recurrence is approximately $(0.3017)(1.9052)^{\;k}$. This is shown in ``On Variants of Diffusion", T. Mullen, PhD Thesis.

\section{$p_2$-Configurations on Paths}

We have already calculated the number of $p_2$-orientations that exist on a path. Now, we will calculate, given a $p_2$-orientation, the number of $p_2$-configurations that exist. For each vertex, we determine the number of possible stack sizes that that vertex can have. We call this number the \textit{multiplier} of that vertex. We now look at an example:

\begin{exmp} \label{exmp:multiplierexmp}
	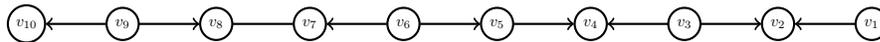
\begin{figure}[h]
		\[
		\begin{tikzpicture}[-,-=stealth', auto,node distance=1.5cm,
		thick,scale=0.6, main node/.style={scale=0.6,circle,draw}]
		
		\node[main node] (1) 					    {$v_{10}$};						 
		\node[main node] (2)  [right=0.8cm of 1]        {$v_9$};
		\node[main node] (3)  [right=0.8cm of 2]        {$v_8$};  
		\node[main node] (4) 	[right=0.8cm of 3]        {$v_7$};						 
		\node[main node] (5)  [right=0.8cm of 4]        {$v_6$}; 
		\node[main node] (6) 	[right=0.8cm of 5]	    {$v_5$};						 
		\node[main node] (7)  [right=0.8cm of 6]        {$v_4$};
		\node[main node] (8)  [right=0.8cm of 7]        {$v_3$};  
		\node[main node] (9) 	[right=0.8cm of 8]        {$v_2$};						 
		\node[main node] (10) [right=0.8cm of 9]        {$v_{1}$}; 
		
		\path[every node/.style={font=\sffamily\small}]

		(4) edge node [] {} (3);
		
		\draw [->] (2) edge (1) (2) edge (3) (5) edge (4)
		(5) edge (6) (6) edge (7) (8) edge (7) (8) edge (9)
		(10) edge (9);

		\end{tikzpicture} \]
		\caption{$P_{10}$ under orientation $R$}
		\label{fig:config}
	\end{figure} 
	\FloatBarrier
	
	We assign $P_{10}$ to have configuration $R$ pictured in Figure~\ref{fig:config}. Moving from right to left, we determine the number of possible stack sizes each vertex can take on:
	
	\begin{enumerate}
		
		\item We fix $v_1$ at 0 chips by convention. So, $v_1$ has a multiplier of 1.
		
		\item We have that $v_2$ must have a negative stack size (so as to receive from $v_1$ in the initial firing) that is large enough to already be in the period. We know that the stack size of $v_1$ will decrease by 1 in the first step and the stack size of $v_2$ will increase by 2 in the first step. This means that $|v_2|_0 < |v_1|_0 = 0$ and 
		$|v_2|_0 + 2 > |v_1|_0 - 1 = -1$. So, $0 > |v_2|_0 > -3$. Thus, the two possible values that $|v_2|_0$ can take on are $-1$ and $-2$. So, $v_2$ has a multiplier of 2.
		
		\item Given a value for $|v_2|_0$, we calculate the number of possible initial stack sizes that $v_3$ can take on. We know $|v_3|_0 > |v_2|_0$ and $|v_3|_0 - 2 < |v_2|_0 + 2$. So, we have that $|v_2|_0 < |v_3|_0 < |v_2|_0 + 4$. Thus, the three possible initial stack sizes for $v_3$ are $|v_2|_0 + 1$, $|v_2|_0 + 2$, and $|v_2|_0 + 3$. So, $v_3$ has a multiplier of 3.    
		
		\item Given a value for $|v_3|_0$, we calculate the number of possible initial stack sizes that $v_4$ can take on. We know $|v_4|_0 < |v_3|_0$ and $|v_4|_0 + 2 > |v_3|_0 - 2$. So, we have that $|v_3|_0 > |v_4|_0 > |v_3|_0 - 4$. Thus, the three possible initial stack sizes for $v_4$ are $|v_3|_0 - 1$, $|v_3|_0 - 2$, and $|v_3|_0 - 3$. So, $v_4$ has a multiplier of 3.   
		
		\item Given a value for $|v_4|_0$, we calculate the number of possible initial stack sizes that $v_5$ can take on. We know $|v_5|_0 > |v_4|_0$ and $|v_5|_0 + 1 - 1 < |v_4|_0 + 2$. So, we have that $|v_4|_0 < |v_5|_0 < |v_4|_0 + 2$. Thus, the only possible initial stack size for $v_5$ is $|v_4|_0 + 1$.  So, $v_5$ has a multiplier of 1.
		
		\item Given a value for $|v_5|_0$, we calculate the number of possible initial stack sizes that $v_6$ can take on. We know $|v_6|_0 > |v_5|_0$ and $|v_6|_0 - 2 < |v_5|_0 + 1 - 1$. So, we have that $|v_5|_0 + 2 > |v_6|_0 > |v_5|_0$. Thus, the only possible initial stack size for $v_6$ is $|v_5|_0 + 1$.  So, $v_6$ has a multiplier of 1.
		
		\item Given a value for $|v_6|_0$, we calculate the number of possible initial stack sizes that $v_7$ can take on. We know that $|v_7|_0 < |v_6|_0$ and $|v_7|_0 + 1 > |v_6|_0 - 2$. So, we have that $|v_6|_0 - 3 < |v_7|_0 < |v_6|_0$. Thus, the two possible initial stack sizes for $v_7$ are $|v_6|_0 - 2$ and $|v_6|_0 - 1$. So, $v_7$ has a multiplier of 2.
		
		\item Given a value for $|v_7|_0$, we calculate the number of possible initial stack sizes that $v_8$ can take on. We know $|v_8|_0 = |v_7|_0$. Thus, the only possible initial stack size for $v_8$ is $|v_7|_0$. So, $v_8$ has a multiplier of 1.
		
		\item Given a value for $|v_8|_0$, we calculate the number of possible initial stack sizes that $v_9$ can take on. We know that $|v_9|_0 > |v_8|_0$ and $|v_9| - 2 < |v_8| + 1$. So, we have that $|v_8|_0 < |v_9|_0 < |v_8| + 3$. Thus, the two possible initial stack sizes for $v_9$ are $|v_8|_0 + 1$ and $|v_8|_0 + 2$.  So, $v_9$ has a multiplier of 2.
		
		\item Given a value for $|v_9|_0$, we calculate the number of possible initial stack sizes that $v_{10}$ can take on. We know that $|v_{10}|_0 < |v_9|_0$ and $|v_{10}|_0 + 1 > v_9 - 2$. So, we have that $|v_9|_0 - 3 < |v_{10}|_0 < |v_9|_0$. Thus, the two possible initial stack sizes for $v_{10}$ are $|v_9|_0 - 2$ and $|v_9|_0 - 1$. Thus, the multiplier for $v_{10}$ is 2. So, $v_10$ has a multiplier of 2.
		
	\end{enumerate}
	
	Multiplying all of these possibilities together we get $1 \times 2 \times 3 \times 3 \times 1 \times 1 \times 2 \times 1 \times 2 \times 2 = 144$ period configurations on this period orientation.
	
\end{exmp}

We now formally define the multiplier of a vertex.

\begin{definition}
	Given a graph orientation $R$ on a path $P_n$, the \textbf{multiplier} assigned to a vertex $v$ represents the number of possible initial stack sizes $v$ could have in a $p_2$-orientation, supposing that an initial stack size has already been chosen for every vertex to the right of $v$.
\end{definition}

Given an assignment of stack sizes to the vertices $v_1, v_2, \cdots v_{i-1}$, the multiplier of $v_i$ in $R$ is the number of different stack sizes $v_i$ can have in a $p_2$-configuration which induces $R$.

Since we are dealing with paths and we are assuming that we are already inside the period of a configuration with period 2, these calculations can be conducted locally, as is evidenced by Example~\ref{exmp:multiplierexmp}. That is, the multiplier of a vertex $v_k$ depends only on the orientation of the edges incident to $v_k$ and those incident to $v_{k-1}$.

Our goal now is to determine the multiplier for any vertex $v_k$ in any $p_2$-configuration on a path.

In order to determine the multiplier of a given vertex $v_k$ in a path $P_n$, we would like to be able to assume that $v_k$ and $v_{k-1}$ are each incident with two edges. We will begin with a smaller theorem that deals with calculating the multiplier for vertices in which this assumption fails.

\begin{theorem} \label{thm:littlemultiplier} (Little Multiplier Theorem) Let $P_n=v_1e_1v_2e_2 \dots e_{n-1}v_n$ be a path on $n \geq 3$ vertices and let $R$ be a $p_2$-orientation on $P_n$. Then
	
	\begin{itemize}

		\item $v_1$ has a multiplier of 1. 
		\item If $e_2$ is flat, then the multiplier of $v_2$ is 1. 
		\item If $e_2$ is directed, then the multiplier of $v_2$ is 2. 
		\item If $e_{n-2}$ is flat, then the multiplier of $v_n$ is 1.
		\item If $e_{n-2}$ is directed, then the multiplier of $v_n$ is 2.
		
	\end{itemize}
	
\end{theorem} 

\begin{proof}
	
	The multiplier for $v_1$ is always 1 because, by convention, we set $v_1$ at 0 chips.
	By Lemma~\ref{lem:endpoint}, we know that $e_1$ cannot be flat. By Lemma~\ref{lem:agreeingarrowsbookend}, we know that $e_2$ does not agree with the $e_1$. So, when calculating the multiplier for $v_2$, there are two cases. Either $e_2$ disagrees with $e_1$, or $e_2$ is flat. That is, we can exclude the following suborientations

	\begin{tikzpicture}[-,-=stealth', auto,node distance=3cm,
	thick,scale=0.6, main node/.style={scale=0.6,circle,draw}]                               
	\node[main node] (1) 					    {$v_1$};						 
	\node[main node] (2)  [right of=1]        {$v_2$};
	\node[main node] (3)  [right of=2]        {$v_3$};  
	\node[main node] (4) 	[below of=1]	    {$v_1$};						 
	\node[main node] (5)  [right of=4]        {$v_2$};
	\node[main node] (6)  [right of=5]        {$v_3$}; 
	\node[main node] (7) 	[below of=4]				    {$v_1$};						 
	\node[main node] (8)  [right of=7]        {$v_2$};
	\node[main node] (9)  [right of=8]        {$v_3$};  
	\node[main node] (10) 	[below of=7]	    {$v_1$};						 
	\node[main node] (11)  [right of=10]        {$v_2$};
	\node[main node] (12)  [right of=11]        {$v_3$};   
	\node[main node] (13) 	[below of=10]	    {$v_1$};						 
	\node[main node] (14)  [right of=13]        {$v_2$};
	\node[main node] (15)  [right of=14]        {$v_3$}; 
	
	\path
	(2) edge (3)
	(1) edge (2)
	(5) edge (6)
	(8) edge (9)
	;	
	
	\draw [->] (4) edge (5) (8) edge (7) (10) edge (11) (11) edge (12) (14) edge (13) (15) edge (14);

	\end{tikzpicture}
	
	We will suppose first that $e_2$ is flat. There are two possibilities.
	
	\begin{enumerate}[(i)]
		
		\item \textit{$e_1$ is directed right.}
		
		\begin{tikzpicture}[-,-=stealth', auto,node distance=3cm,
		thick,scale=0.6, main node/.style={scale=0.6,circle,draw}]                               
		\node[main node] (1) 					    {$v_1$};						 
		\node[main node] (2)  [right of=1]        {$v_2$};
		\node[main node] (3)  [right of=2]        {$v_3$};

		\path
		(1) edge (2)
		
		;	
		
		\draw [->] (2) edge (3);

		\end{tikzpicture}

		The net effect of the initial firing on $v_2$ is a decrease of one chip, and the net effect of the initial firing on $v_{1}$ is an increase of one chip. This means that $|v_2|_0 > |v_{1}|_0$ and $|v_2|_0 - 1 < |v_{1}|_0 + 1$. So, $|v_{1}|_0 < |v_2|_0 < |v_{1}|_0 + 2$. Therefore, $|v_{1}|_0 + 1$ is the only possible initial stack sizes for $v_2$. Since there is only one possible initial stack size, we say that $v_2$ has a multiplier of 1.

		\item \textit{ $e_1$ is directed left.} 
		
		\begin{tikzpicture}[-,-=stealth', auto,node distance=3cm,
		thick,scale=0.6, main node/.style={scale=0.6,circle,draw}]                               
		\node[main node] (1) 					    {$v_1$};						 
		\node[main node] (2)  [right of=1]        {$v_2$};
		\node[main node] (3)  [right of=2]        {$v_3$};

		\path
		(1) edge (2)
		
		;	
		
		\draw [->] (3) edge (2);

		\end{tikzpicture}

		The net effect of the initial firing on $v_2$ is an increase of one chip, and the net effect of the initial firing on $v_{1}$ is a decrease of one chip. This means that $|v_2|_0 < |v_{1}|_0$ and $|v_2|_0 + 1 > |v_{1}|_0 - 1$. So, $|v_{1}|_0 > |v_2|_0 > |v_{1}|_0 - 2$. Therefore, $|v_{1}|_0 - 1$ is the only possible initial stack sizes for $v_2$. Since there is only one possible initial stack size, we say that $v_2$ has a multiplier of 1.
		
	\end{enumerate}

	We now suppose instead that $e_2$ disagrees with $e_1$. There are two possibilities.
	
	\begin{enumerate}[(i)]
		
		\item \textit{$e_2$ is directed right.}
		
		\begin{tikzpicture}[-,-=stealth', auto,node distance=3cm,
		thick,scale=0.6, main node/.style={scale=0.6,circle,draw}]                               
		\node[main node] (1) 					    {$v_1$};						 
		\node[main node] (2)  [right of=1]        {$v_2$};
		\node[main node] (3)  [right of=2]        {$v_3$};

		\path
		
		;	
		
		\draw [->] (1) edge (2) (3) edge (2);

		\end{tikzpicture}

		The net effect of the initial firing on $v_2$ is an increase of two chips, and the net effect of the initial firing on $v_{1}$ is a decrease of one chip. This means that $|v_2|_0 < |v_{1}|_0$ and $|v_2|_0 + 2 > |v_{1}|_0 - 1$. So, $|v_{1}|_0 > |v_2|_0 > |v_{1}|_0 - 3$. Therefore, $|v_{1}|_0 - 1$ and $|v_1|_0 - 2$ are the only possible initial stack sizes for $v_2$. Since there are only two possible initial stack sizes, we say that $v_2$ has a multiplier of 2.

		\item \textit{$e_2$ is directed left.}
		
		\begin{tikzpicture}[-,-=stealth', auto,node distance=3cm,
		thick,scale=0.6, main node/.style={scale=0.6,circle,draw}]                               
		\node[main node] (1) 					    {$v_1$};						 
		\node[main node] (2)  [right of=1]        {$v_2$};
		\node[main node] (3)  [right of=2]        {$v_3$};

		\path
		
		;	
		
		\draw [->] (2) edge (1) (2) edge (3);

		\end{tikzpicture}

		The net effect of the initial firing on $v_2$ is a decrease of two chips, and the net effect of the initial firing on $v_{1}$ is an increase of one chip. This means that $|v_2|_0 > |v_{1}|_0$ and $|v_2|_0 - 2 < |v_{1}|_0 + 1$. So, $|v_{1}|_0 < |v_2|_0 < |v_{1}|_0 + 3$. Therefore, $|v_{1}|_0 + 1$ and $|v_1|_0 + 2$ are the only possible initial stack sizes for $v_2$. Since there are only two possible initial stack sizes, we say that $v_2$ has a multiplier of 2.
		
	\end{enumerate}
	
	We now turn our attention to $v_n$.
	By Lemma~\ref{lem:endpoint}, we know that the edge $e_{n-1}$ is not flat. By Lemma~\ref{lem:agreeingarrowsbookend}, we know that the edge $e_{n-2}$ does not agree with the edge $e_{n-1}$. So, when calculating the multiplier for $v_n$, there are two cases. Either $e_{n-2}$ disagrees with $e_{n-1}$ or $e_{n-2}$ is flat. That is, we can exclude the following suborientations

	
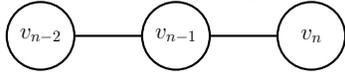
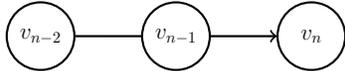
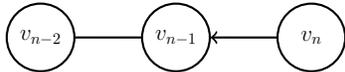
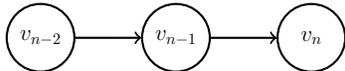
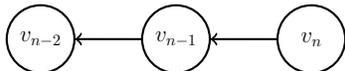
\begin{tikzpicture}[-,-=stealth', auto,node distance=3cm,
	thick,scale=0.6, main node/.style={scale=0.6,circle,draw, minimum size=1.5cm, font=\sffamily\Large\bfseries}]                               
	\node[main node] (1) 					    {$v_{n-2}$};						 
	\node[main node] (2)  [right of=1]        {$v_{n-1}$};
	\node[main node] (3)  [right of=2]        {$v_n$};  
	\node[main node] (4) 	[below of=1]	    {$v_{n-2}$};						 
	\node[main node] (5)  [right of=4]        {$v_{n-1}$};
	\node[main node] (6)  [right of=5]        {$v_n$}; 
	\node[main node] (7) 	[below of=4]				    {$v_{n-2}$};						 
	\node[main node] (8)  [right of=7]        {$v_{n-1}$};
	\node[main node] (9)  [right of=8]        {$v_n$};  
	\node[main node] (10) 	[below of=7]	    {$v_{n-2}$};						 
	\node[main node] (11)  [right of=10]        {$v_{n-1}$};
	\node[main node] (12)  [right of=11]        {$v_n$};   
	\node[main node] (13) 	[below of=10]	    {$v_{n-2}$};						 
	\node[main node] (14)  [right of=13]        {$v_{n-1}$};
	\node[main node] (15)  [right of=14]        {$v_n$}; 
	
	\path
	(2) edge (3)
	(1) edge (2)
	(5) edge (4)
	(8) edge (7)
	;	
	
	\draw [->] (5) edge (6) (9) edge (8) (10) edge (11) (11) edge (12) (14) edge (13) (15) edge (14);

	\end{tikzpicture}

	We will suppose first that $e_{n-2}$ is flat. There are two possibilities.
	
	\begin{enumerate}[(i)]
		\item \textit{$e_{n-1}$ is directed right.}
		
		
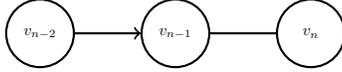
\begin{tikzpicture}[-,-=stealth', auto,node distance=3cm,
		thick,scale=0.6, main node/.style={scale=0.6,circle,draw, minimum size=1.5cm}]                               
		\node[main node] (1) 					    {$v_{n-2}$};						 
		\node[main node] (2)  [right of=1]        {$v_{n-1}$};
		\node[main node] (3)  [right of=2]        {$v_n$};

		\path
		(3) edge (2)
		
		;	
		
		\draw [->] (1) edge (2);

		\end{tikzpicture}

		The net effect of the initial firing on $v_n$ is a decrease of one chip, and the net effect of the initial firing on $v_{n-1}$ is an increase of one chip. This means that $|v_n|_0 > |v_{n-1}|_0$ and $|v_n|_0 - 1 < |v_{n-1}|_0 + 1$. So, $|v_{n-1}|_0 < |v_n|_0 < |v_{n-1}|_0 + 2$. Therefore, $|v_{n-1}|_0 + 1$ is the only possible initial stack size for $v_n$. Since there is only one possible initial stack size, we say that $v_n$ has a multiplier of 1.

		\item \textit{$e_{n-1}$ is directed left.} 
		
		
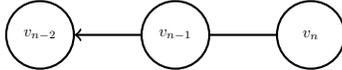
\begin{tikzpicture}[-,-=stealth', auto,node distance=3cm,
		thick,scale=0.6, main node/.style={scale=0.6,circle,draw, minimum size=1.5cm}]                               
		\node[main node] (1) 					    {$v_{n-2}$};						 
		\node[main node] (2)  [right of=1]        {$v_{n-1}$};
		\node[main node] (3)  [right of=2]        {$v_n$};

		\path
		(3) edge (2)
		
		;	
		
		\draw [->] (2) edge (1);

		\end{tikzpicture}

		The net effect of the initial firing on $v_n$ is an increase of one chip, and the net effect of the initial firing on $v_{n-1}$ is a decrease of one chip. This means that $|v_n|_0 < |v_{n-1}|_0$ and $|v_n|_0 + 1 > |v_{n-1}|_0 - 1$. So, $|v_{n-1}|_0 > |v_n|_0 > |v_{n-1}|_0 - 2$. Therefore, $|v_{n-1}|_0 - 1$ is the only possible initial stack size for $v_n$. Since there is only one possible initial stack size, we say that $v_n$ has a multiplier of 1.
		
	\end{enumerate}

	We now suppose instead that $e_n$ disagrees with $e_{n-1}$. There are two possibilities.
	
	\begin{enumerate}
		
		\item \textit{$e_n$ is directed right.}
		
		
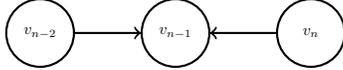
\begin{tikzpicture}[-,-=stealth', auto,node distance=3cm,
		thick,scale=0.6, main node/.style={scale=0.6,circle,draw, minimum size=1.5cm}]                               
		\node[main node] (1) 					    {$v_{n-2}$};						 
		\node[main node] (2)  [right of=1]        {$v_{n-1}$};
		\node[main node] (3)  [right of=2]        {$v_n$};

		\path
		
		;	
		
		\draw [->] (1) edge (2) (3) edge (2);

		\end{tikzpicture}

		The net effect of the initial firing on $v_n$ is a decrease of one chip, and the net effect of the initial firing on $v_{n-1}$ is an increase of two chips. This means that $|v_n|_0 > |v_{n-1}|_0$ and $|v_n|_0 - 1 < |v_{n-1}|_0 + 2$. So, $|v_{n-1}|_0 < |v_n|_0 < |v_{n-1}|_0 + 3$. Therefore, $|v_{n-1}|_0 + 1$ and $|v_1|_0 + 2$ are the only possible initial stack sizes for $v_n$. Since there are only two possible initial stack sizes, we say that $v_n$ has a multiplier of 2.

		\item \textit{$e_n$ is directed left.}
		
		
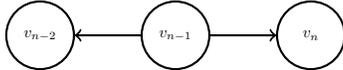
\begin{tikzpicture}[-,-=stealth', auto,node distance=3cm,
		thick,scale=0.6, main node/.style={scale=0.6,circle,draw, minimum size=1.5cm}]                               
		\node[main node] (1) 					    {$v_{n-2}$};						 
		\node[main node] (2)  [right of=1]        {$v_{n-1}$};
		\node[main node] (3)  [right of=2]        {$v_n$};

		\path
		
		;	
		
		\draw [->] (2) edge (1) (2) edge (3);

		\end{tikzpicture}

		The net effect of the initial firing on $v_n$ is an increase of one chip, and the net effect of the initial firing on $v_{n-1}$ is a decrease of two chips. This means that $|v_n|_0 < |v_{n-1}|_0$ and $|v_n|_0 + 1 > |v_{n-1}|_0 - 2$. So, $|v_{n-1}|_0 > |v_n|_0 > |v_{1}|_0 - 3$. Therefore, $|v_{n-1}|_0 - 1$ and $|v_{n-1}|_0 - 2$ are the only possible initial stack sizes for $v_n$. Since there are only two possible initial stack sizes, we say that $v_n$ has a multiplier of 2.
		
	\end{enumerate}
	
\end{proof}

We will now look at the multipliers of the other vertices.

\begin{theorem} (The Multiplier Theorem)\label{thm:multiplier}
	Let $R$ be a $p_2$-orientation on a path 
	
	$P_n=v_1e_1v_2e_2 \dots e_{n-1}v_n$ with $n \geq 4$. If a vertex, $v_k$, and its neighbour, $v_{k-1}$, each have exactly two neighbours, then the multiplier of $v_k$ is 1, 2, or 3 depending on the suborientation within which it exists, as outlined in Table~\ref{tab:multipliers}. 
\end{theorem}

\begin{table}[h!] 
	\centering
	\begin{tabular}{ |c|c|c|c|c|c|}
		\hline
		& & & &\\
		\hline
		3  & 
		
		\begin{tikzpicture}[-,-=stealth', auto,node distance=1.5cm,
		thick,scale=0.6, main node/.style={scale=0.6,circle,draw, minimum size=1cm,font=\sffamily\Large\bfseries}]                               
		\node[main node] (1) 					    {};						 
		\node[main node] (2)  [right of=1]        {$v_k$};
		\node[main node] (3)  [right of=2]        {};  
		\node[main node] (4) 	[right of=3]        {};						 
		\node[main node] (5) 	[below of=1]		{};						 
		\node[main node] (6)  [right of=5]        {$v_k$};
		\node[main node] (7)  [right of=6]        {};  
		\node[main node] (8) 	[right of=7]        {};

		\draw [->] (2) edge (1) (2) edge (3) (4) edge (3)
		(5) edge (6) (7) edge (6) (7) edge (8);

		\end{tikzpicture} 
		& & &
		
		\\ 
		\hline
		2 & 
		
		\begin{tikzpicture}[-,-=stealth', auto,node distance=1.5cm,
		thick,scale=0.6, main node/.style={scale=0.6,circle,draw, minimum size=1cm, font=\sffamily\Large\bfseries}]                               
		\node[main node] (1) 					    {};						 
		\node[main node] (2)  [right of=1]        {$v_k$};
		\node[main node] (3)  [right of=2]        {};  
		\node[main node] (4) 	[right of=3]        {};						 
		\node[main node] (5)  [below of=1]        {};
		\node[main node] (6)  [right of=5]        {$v_k$};
		\node[main node] (7)  [right of=6]        {};
		\node[main node] (8)  [right of=7]        {};

		\draw [->] (2) edge (1) (2) edge (3) 
		(5) edge (6) (7) edge (6);
		
		\path[every node/.style={font=\sffamily\small}]		 
		(3) edge node [] {} (4)
		(7) edge node [] {} (8);
		
		\end{tikzpicture} 
		
		& 
		
		\begin{tikzpicture}[-,-=stealth', auto,node distance=1.5cm,
		thick,scale=0.6, main node/.style={scale=0.6,circle,draw, minimum size=1cm, font=\sffamily\Large\bfseries}]                               
		\node[main node] (1) 					    {};						 
		\node[main node] (2)  [right of=1]        {$v_k$};
		\node[main node] (3)  [right of=2]        {};  
		\node[main node] (4) 	[right of=3]        {};						 
		\node[main node] (5)  [below of=1]        {};
		\node[main node] (6)  [right of=5]        {$v_k$};
		\node[main node] (7)  [right of=6]        {};
		\node[main node] (8)  [right of=7]        {};  
		
		\draw [->] (2) edge (3) (4) edge (3)  
		(7) edge (6) (7) edge (8);
		
		\path[every node/.style={font=\sffamily\small}]		 
		(1) edge node [] {} (2)
		(5) edge node [] {} (6);
		
		\end{tikzpicture} 
		
		& & 
		
		\\
		\hline
		1  &

		\begin{tikzpicture}[-,-=stealth', auto,node distance=1.5cm,
		thick,scale=0.6, main node/.style={scale=0.6,circle,draw, minimum size=1cm, font=\sffamily\Large\bfseries}]                               
		\node[main node] (1) 					    {};						 
		\node[main node] (2)  [right of=1]        {$v_k$};
		\node[main node] (3)  [right of=2]        {};  
		\node[main node] (4) 	[right of=3]        {};						 
		\node[main node] (5)  [below of=1]        {};
		\node[main node] (6)  [right of=5]        {$v_k$};
		\node[main node] (7)  [right of=6]        {};
		\node[main node] (8)  [right of=7]        {};
		
		\draw [->] (2) edge (1) (2) edge (3) (3) edge (4) 
		(5) edge (6) (7) edge (6) (8) edge (7);

		\end{tikzpicture} 
		
		&
		
		\begin{tikzpicture}[-,-=stealth', auto,node distance=1.5cm,
		thick,scale=0.6, main node/.style={scale=0.6,circle,draw, minimum size=1cm, font=\sffamily\Large\bfseries}]                               
		\node[main node] (1) 					    {};						 
		\node[main node] (2)  [right of=1]        {$v_k$};
		\node[main node] (3)  [right of=2]        {};  
		\node[main node] (4) 	[right of=3]        {};						 
		\node[main node] (5)  [below of=1]        {};
		\node[main node] (6)  [right of=5]        {$v_k$};
		\node[main node] (7)  [right of=6]        {};
		\node[main node] (8)  [right of=7]        {};

		\draw [->] (2) edge (1) (3) edge (2) (3) edge (4) 
		(5) edge (6) (6) edge (7) (8) edge (7);

		\end{tikzpicture} 
		
		& 
		
		\begin{tikzpicture}[-,-=stealth', auto,node distance=1.5cm,
		thick,scale=0.6, main node/.style={scale=0.6,circle,draw, minimum size=1cm, font=\sffamily\Large\bfseries}]                               
		\node[main node] (1) 					    {};						 
		\node[main node] (2)  [right of=1]        {$v_k$};
		\node[main node] (3)  [right of=2]        {};  
		\node[main node] (4) 	[right of=3]        {};						 
		\node[main node] (5)  [below of=1]        {};
		\node[main node] (6)  [right of=5]        {$v_k$};
		\node[main node] (7)  [right of=6]        {};
		\node[main node] (8)  [right of=7]        {};
		
		\draw [->] (2) edge (3) 
		(7) edge (6);
		
		\path[every node/.style={font=\sffamily\small}]		 
		(1) edge node [] {} (2)
		
		(3) edge node [] {} (4)
		
		(5) edge node [] {} (6)
		
		(7) edge node [] {} (8);
		
		\end{tikzpicture} 
		
		&  
		
		\begin{tikzpicture}[-,-=stealth', auto,node distance=1.5cm,
		thick,scale=0.6, main node/.style={scale=0.6,circle,draw, minimum size=1cm, font=\sffamily\Large\bfseries}]                               
		\node[main node] (1) 					    {};						 
		\node[main node] (2)  [right of=1]        {$v_k$};
		\node[main node] (3)  [right of=2]        {};  
		\node[main node] (4) 	[right of=3]        {};						 
		\node[main node] (5) 	[below of=1]	    {};						 
		\node[main node] (6)  [right of=5]        {$v_k$};
		\node[main node] (7)  [right of=6]        {};  
		\node[main node] (8) 	[right of=7]        {};						    
		
		\draw [->] (2) edge (1) (3) edge (4)
		(5) edge (6) (8) edge (7) ;
		
		\path[every node/.style={font=\sffamily\small}]		 
		(2) edge node [] {} (3)
		(6) edge node [] {} (7);
		\end{tikzpicture} 
		
		\\ 
		\hline
	\end{tabular}
	\caption{Multipliers (listed in the leftmost column) of $v_k$ based on neighbourhood}
	\label{tab:multipliers}
\end{table}
\FloatBarrier

\begin{proof}
	
	We will begin by proving that no suborientation omitted from Table~\ref{tab:multipliers} can be contained within a $p_2$-orientation.
	
	Every edge has 3 possible orientations. Therefore, there exist $3^3 = 27$ graph orientations of $P_4$. However, we know several of these orientations cannot exist as suborientations within a $p_2$-orientation by Theorem~\ref{thm:illegalorientations}. We now list these orientations which cannot exist within a $p_2$-orientation.
	
	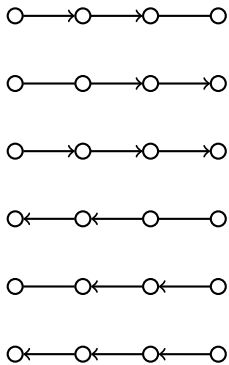
\begin{figure}
		
		The following 5 suborientations cannot exist within a $p_2$-orientation by Lemma~\ref{lem:noflat}. 
		
		
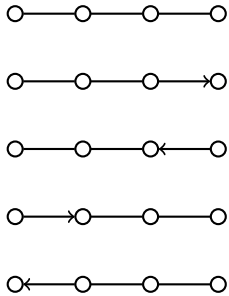
\vspace{1cm} 
		
		\begin{tikzpicture}[-,-=stealth', auto,node distance=1.5cm,
		thick,scale=0.6, main node/.style={scale=0.6,circle,draw}]                               
		\node[main node] (1) 					    {};						 
		\node[main node] (2)  [right of=1]        {};
		\node[main node] (3)  [right of=2]        {};  
		\node[main node] (4) 	[right of=3]        {};						 
		\node[main node] (5) 	[below of=1]	    {};						 
		\node[main node] (6)  [right of=5]        {};
		\node[main node] (7)  [right of=6]        {};  
		\node[main node] (8) 	[right of=7]        {};						 
		\node[main node] (9) 	[below of=5]	    {};						 
		\node[main node] (10) [right of=9]        {};
		\node[main node] (11) [right of=10]       {};  
		\node[main node] (12) [right of=11]       {};						 
		\node[main node] (13) [below of=9]		{};						 
		\node[main node] (14) [right of=13]       {};
		\node[main node] (15) [right of=14]       {};  
		\node[main node] (16) [right of=15]       {};						  
		\node[main node] (17) [below of=13]		{};						 
		\node[main node] (18) [right of=17]       {};
		\node[main node] (19) [right of=18]       {};  
		\node[main node] (20) [right of=19]       {};						  
		
		\path
		
		(1) edge (2)
		(2) edge (3)
		(3) edge (4)
		(5) edge (6)
		(6) edge (7)
		(9) edge (10)
		(10) edge (11)
		(14) edge (15)
		(15) edge (16)	
		(18) edge (19)
		(19) edge (20);	
		
		\draw [->] (7) edge (8) (12) edge (11) (13) edge (14) (18) edge (17);

		\end{tikzpicture}
		
		\vspace{1cm}
		
		The following 2 suborientations cannot exist within a $p_2$-orientation by Lemma~\ref{lem:rfl}.
		
		
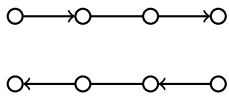
\vspace{1cm}
		
		\begin{tikzpicture}[-,-=stealth', auto,node distance=1.5cm,
		thick,scale=0.6, main node/.style={scale=0.6,circle,draw}]                               
		\node[main node] (1) 					    {};						 
		\node[main node] (2)  [right of=1]        {};
		\node[main node] (3)  [right of=2]        {};  
		\node[main node] (4) 	[right of=3]        {};						 
		\node[main node] (5) 	[below of=1]	    {};						 
		\node[main node] (6)  [right of=5]        {};
		\node[main node] (7)  [right of=6]        {};  
		\node[main node] (8) 	[right of=7]        {};

		\path
		(2) edge (3)
		(6) edge (7);	
		
		\draw [->] (1) edge (2) (3) edge (4) (6) edge (5) (8) edge (7);

		\end{tikzpicture}
		
		\vspace{1cm}
		
		The following 6 suborientations cannot exist within a $p_2$-orientation by Lemma~\ref{lem:agreeingarrowsbookend}.
		
		\vspace{1cm}
		
		\begin{tikzpicture}[-,-=stealth', auto,node distance=1.5cm,
		thick,scale=0.6, main node/.style={scale=0.6,circle,draw}]                               
		\node[main node] (1) 					    {};						 
		\node[main node] (2)  [right of=1]        {};
		\node[main node] (3)  [right of=2]        {};  
		\node[main node] (4) 	[right of=3]        {};						 
		\node[main node] (5) 	[below of=1]	    {};						 
		\node[main node] (6)  [right of=5]        {};
		\node[main node] (7)  [right of=6]        {};  
		\node[main node] (8) 	[right of=7]        {};						 
		\node[main node] (9) 	[below of=5]	    {};						 
		\node[main node] (10) [right of=9]        {};
		\node[main node] (11) [right of=10]       {};  
		\node[main node] (12) [right of=11]       {};						 
		\node[main node] (13) [below of=9]		{};						 
		\node[main node] (14) [right of=13]       {};
		\node[main node] (15) [right of=14]       {};  
		\node[main node] (16) [right of=15]       {};						  
		\node[main node] (17) [below of=13]		{};						 
		\node[main node] (18) [right of=17]       {};
		\node[main node] (19) [right of=18]       {};  
		\node[main node] (20) [right of=19]       {};						  
		\node[main node] (21) [below of=17]		{};						 
		\node[main node] (22) [right of=21]       {};
		\node[main node] (23) [right of=22]       {};  
		\node[main node] (24) [right of=23]       {};						  
		
		\path
		(3) edge (4)
		(5) edge (6)
		(15) edge (16)	
		(18) edge (17);	
		
		\draw [->] (1) edge (2) (2) edge (3) (6) edge (7) (7) edge (8) (9) edge (10) (10) edge (11) (11) edge (12)
		(14) edge (13) (15) edge (14) (19) edge (18) (20) edge (19) (22) edge (21) (23) edge (22) (24) edge (23);

		\end{tikzpicture}
		\caption{List of suborientations which cannot exist within a $p_2$-orientation}
		\label{fig:illegalsuborientations}
	\end{figure}
	
	\FloatBarrier

	What remains are the $27 - 13 = 14$ suborientations listed in Table~\ref{tab:multipliers}. We will break these 14 suborientations into 7 pairs of suborientations and show their multipliers using a case analysis. Each orientation will be paired with the orientation created by reversing the direction of every directed edge contained within. We will see that these pairs always have the same multiplier and can be proven using similar arguments. Note that by Corollary~\ref{cor:longcor}, every period that contains one of these orientations must also contain the one with which it is paired. 
	
	\textbf{Case 1:} Alternating arrow suborientation.
	
	
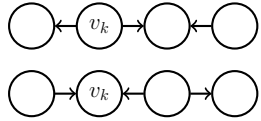
\begin{figure}[H]
		\[
		\begin{tikzpicture}[-,-=stealth', auto,node distance=1.5cm,
		thick,scale=0.6, main node/.style={scale=0.6,circle,draw, minimum size=1cm, font=\sffamily\Large\bfseries}]                               
		\node[main node] (1) 					    {};						 
		\node[main node] (2)  [right of=1]        {$v_k$};
		\node[main node] (3)  [right of=2]        {};  
		\node[main node] (4) 	[right of=3]        {};						 
		\node[main node] (5) 	[below of=1]		{};						 
		\node[main node] (6)  [right of=5]        {$v_k$};
		\node[main node] (7)  [right of=6]        {};  
		\node[main node] (8) 	[right of=7]        {};

		\draw [->] (2) edge (1) (2) edge (3) (4) edge (3)
		(5) edge (6) (7) edge (6) (7) edge (8);

		\end{tikzpicture} \]
	\end{figure}
	\FloatBarrier
	
	First assume that $v_k$ is losing two chips in the initial firing. The net effect of the initial firing on $v_k$ is a decrease of two chips, and the net effect of the initial firing on $v_{k-1}$ is an increase of two chips. This means that $|v_k|_0 > |v_{k-1}|_0$ and $|v_k|_0 - 2 < |v_{k-1}|_0 + 2$. So, $|v_{k-1}|_0 < |v_k|_0 < |v_{k-1}|_0 + 4$. Therefore, $|v_{k-1}|_0 + 1$, $|v_{k-1}|_0 + 2$, and $|v_{k-1}|_0 + 3$ are the only possible initial stack sizes for $v_k$. Since there are only three possible initial stack sizes, $v_k$ has a multiplier of 3.  
	
	Now assume that instead, $v_k$ is gaining two chips in the initial firing. The net effect of the initial firing on $v_k$ is an increase of two chips, and the net effect of the initial firing on $v_{k-1}$ is a decrease of two chips. This means that $|v_k|_0 < |v_{k-1}|_0$ and $|v_k|_0 + 2 > |v_{k-1}|_0 - 2$. So, $|v_{k-1}|_0 > |v_k|_0 > |v_{k-1}|_0 - 4$. Therefore, $|v_{k-1}|_0 - 1$, $|v_{k-1}|_0 - 2$, and $|v_{k-1}|_0 - 3$ are the only possible initial stack sizes for $v_k$. Since there are only three possible initial stack sizes, $v_k$ has a multiplier of 3.  
	
	\textbf{Case 2:}
	
	
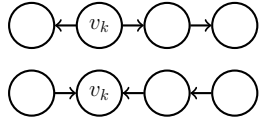
\begin{figure}[H]
		\[
		\begin{tikzpicture}[-,-=stealth', auto,node distance=1.5cm,
		thick,scale=0.6, main node/.style={scale=0.6,circle,draw, minimum size=1cm, font=\sffamily\Large\bfseries}]                               
		\node[main node] (1) 					    {};						 
		\node[main node] (2)  [right of=1]        {$v_k$};
		\node[main node] (3)  [right of=2]        {};  
		\node[main node] (4) 	[right of=3]        {};						 
		\node[main node] (5) 	[below of=1]		{};						 
		\node[main node] (6)  [right of=5]        {$v_k$};
		\node[main node] (7)  [right of=6]        {};  
		\node[main node] (8) 	[right of=7]        {};

		\draw [->] (2) edge (1) (2) edge (3) (3) edge (4)
		(5) edge (6) (7) edge (6) (8) edge (7);

		\end{tikzpicture} \]
	\end{figure}
	\FloatBarrier
	
	First assume that $v_k$ is losing two chips in the initial firing. The net effect of the initial firing on $v_k$ is a decrease of two chips, and the net effect of the initial firing on $v_{k-1}$ is no change in the number of chips. This means that $|v_k|_0 > |v_{k-1}|_0$ and $|v_k|_0 - 2 < |v_{k-1}|_0$. So, $|v_{k-1}|_0 < |v_k|_0 < |v_{k-1}|_0 + 2$. Therefore, $|v_{k-1}|_0 + 1$ is the only possible initial stack size for $v_k$. Since there is only one possible initial stack size, $v_k$ has a multiplier of 1.
	
	Now assume that instead, $v_k$ is gaining two chips in the initial firing. The net effect of the initial firing on $v_k$ is an increase of two chips, and the net effect of the initial firing on $v_{k-1}$ is no change in the number of chips. This means that $|v_k|_0 < |v_{k-1}|_0$ and $|v_k|_0 + 2 > |v_{k-1}|_0$. So, $|v_{k-1}|_0 > |v_k|_0 > |v_{k-1}|_0 - 2$. Therefore, $|v_{k-1}|_0 - 1$ is the only possible initial stack size for $v_k$. Since there is only one possible initial stack size, $v_k$ has a multiplier of 1.
	
	\textbf{Case 3:}
	
	
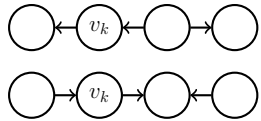
\begin{figure}[H]
		\[
		\begin{tikzpicture}[-,-=stealth', auto,node distance=1.5cm,
		thick,scale=0.6, main node/.style={scale=0.6,circle,draw, minimum size=1cm, font=\sffamily\Large\bfseries}]                               
		\node[main node] (1) 					    {};						 
		\node[main node] (2)  [right of=1]        {$v_k$};
		\node[main node] (3)  [right of=2]        {};  
		\node[main node] (4) 	[right of=3]        {};						 
		\node[main node] (5) 	[below of=1]		{};						 
		\node[main node] (6)  [right of=5]        {$v_k$};
		\node[main node] (7)  [right of=6]        {};  
		\node[main node] (8) 	[right of=7]        {};

		\draw [->] (2) edge (1) (3) edge (2) (3) edge (4)
		(5) edge (6) (6) edge (7) (8) edge (7);

		\end{tikzpicture} \]
	\end{figure}
	\FloatBarrier
	
	First assume that $v_{k-1}$ is losing two chips in the initial firing. The net effect of the initial firing on $v_k$ is no change in the number of chips, and the net effect of the initial firing on $v_{k-1}$ is a decrease of two chips. This means that $|v_k|_0 < |v_{k-1}|_0$ and $|v_k|_0 > |v_{k-1}|_0 - 2$. So, $|v_{k-1}|_0 > |v_k|_0 > |v_{k-1}|_0 - 2$. Therefore, $|v_{k-1}|_0 - 1$ is the only possible initial stack size for $v_k$. Since there is only one possible initial stack size, $v_k$ has a multiplier of 1.
	
	Now assume that instead, $v_{k-1}$ is gaining two chips in the initial firing. The net effect of the initial firing on $v_k$ is no change in the number of chips, and the net effect of the initial firing on $v_{k-1}$ is an increase of two chips. This means that $|v_k|_0 > |v_{k-1}|_0$ and $|v_k|_0 < |v_{k-1}|_0 + 2$. So, $|v_{k-1}|_0 < |v_k|_0 < |v_{k-1}|_0 + 2$. Therefore, $|v_{k-1}|_0 + 1$ is the only possible initial stack size for $v_k$. Since there is only one possible initial stack size, $v_k$ has a multiplier of 1.
	
	\textbf{Case 4:}
	
	
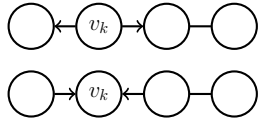
\begin{figure}[H]
		\[
		\begin{tikzpicture}[-,-=stealth', auto,node distance=1.5cm,
		thick,scale=0.6, main node/.style={scale=0.6,circle,draw, minimum size=1cm, font=\sffamily\Large\bfseries}]                               
		\node[main node] (1) 					    {};						 
		\node[main node] (2)  [right of=1]        {$v_k$};
		\node[main node] (3)  [right of=2]        {};  
		\node[main node] (4) 	[right of=3]        {};						 
		\node[main node] (5) 	[below of=1]		{};						 
		\node[main node] (6)  [right of=5]        {$v_k$};
		\node[main node] (7)  [right of=6]        {};  
		\node[main node] (8) 	[right of=7]        {};

		\draw [->] (2) edge (1) (2) edge (3) 
		(5) edge (6) (7) edge (6) ;
		
		\path[every node/.style={font=\sffamily\small}]		 
		(3) edge node [] {} (4)
		(7) edge node [] {} (8);
		
		\end{tikzpicture} \]
	\end{figure}
	\FloatBarrier
	
	First assume that $v_{k}$ is losing two chips in the initial firing. The net effect of the initial firing on $v_k$ is a decrease of two chips, and the net effect of the initial firing on $v_{k-1}$ is an increase of one chip. This means that $|v_k|_0 > |v_{k-1}|_0$ and $|v_k|_0 - 2 < |v_{k-1}|_0 + 1$. So, $|v_{k-1}|_0 < |v_k|_0 < |v_{k-1}|_0 + 3$. Therefore, $|v_{k-1}|_0 + 1$ and $|v_{k-1} + 2$ are the only possible initial stack sizes for $v_k$. Since there are only two possible initial stack sizes, $v_k$ has a multiplier of 2.
	
	Now assume that instead, $v_{k}$ is gaining two chips in the initial firing. The net effect of the initial firing on $v_k$ is an increase of two chips, and the net effect of the initial firing on $v_{k-1}$ is a decrease of one chip. This means that $|v_k|_0 < |v_{k-1}|_0$ and $|v_k|_0 + 2 > |v_{k-1}|_0 - 1$. So, $|v_{k-1}|_0 > |v_k|_0 > |v_{k-1}|_0 - 3$. Therefore, $|v_{k-1}|_0 - 1$ and $|v_{k-1} - 2$ are the only possible initial stack sizes for $v_k$. Since there are only two possible initial stack sizes, $v_k$ has a multiplier of 2.
	
	\textbf{Case 5:}
	
	
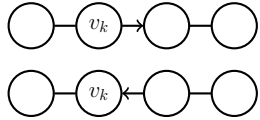
\begin{figure}[H]
		\[
		\begin{tikzpicture}[-,-=stealth', auto,node distance=1.5cm,
		thick,scale=0.6, main node/.style={scale=0.6,circle,draw, minimum size=1cm, font=\sffamily\Large\bfseries}]                               
		\node[main node] (1) 					    {};						 
		\node[main node] (2)  [right of=1]        {$v_k$};
		\node[main node] (3)  [right of=2]        {};  
		\node[main node] (4) 	[right of=3]        {};						 
		\node[main node] (5) 	[below of=1]		{};						 
		\node[main node] (6)  [right of=5]        {$v_k$};
		\node[main node] (7)  [right of=6]        {};  
		\node[main node] (8) 	[right of=7]        {};

		\draw [->] (2) edge (3) 
		(7) edge (6) ;
		
		\path[every node/.style={font=\sffamily\small}]		 
		(1) edge node [] {} (2)
		(5) edge node [] {} (6)
		(3) edge node [] {} (4)
		(7) edge node [] {} (8);
		
		\end{tikzpicture} \]
	\end{figure}
	\FloatBarrier
	
	First assume that $v_{k}$ is losing one chip in the initial firing. The net effect of the initial firing on $v_k$ is a decrease of one chip, and the net effect of the initial firing on $v_{k-1}$ is an increase of one chip. This means that $|v_k|_0 > |v_{k-1}|_0$ and $|v_k|_0 - 1 < |v_{k-1}|_0 + 1$. So, $|v_{k-1}|_0 < |v_k|_0 < |v_{k-1}|_0 + 2$. Therefore, $|v_{k-1}|_0 + 1$ is the only possible initial stack size for $v_k$. Since there is only one possible initial stack size, $v_k$ has a multiplier of 1.
	
	Now assume that instead, $v_{k}$ is gaining one chip in the initial firing. The net effect of the initial firing on $v_k$ is an increase of one chip, and the net effect of the initial firing on $v_{k-1}$ is a decrease of one chip. This means that $|v_k|_0 < |v_{k-1}|_0$ and $|v_k|_0 + 1 > |v_{k-1}|_0 - 1$. So, $|v_{k-1}|_0 > |v_k|_0 > |v_{k-1}|_0 - 2$. Therefore, $|v_{k-1}|_0 - 1$ is the only possible initial stack size for $v_k$. Since there is only one possible initial stack size, $v_k$ has a multiplier of 1.
	
	\textbf{Case 6:}
	
	
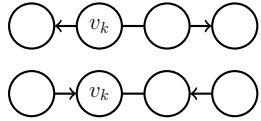
\begin{figure}[H]
		\[
		\begin{tikzpicture}[-,-=stealth', auto,node distance=1.5cm,
		thick,scale=0.6, main node/.style={scale=0.6,circle,draw, minimum size=1cm, font=\sffamily\Large\bfseries}]                               
		\node[main node] (1) 					    {};						 
		\node[main node] (2)  [right of=1]        {$v_k$};
		\node[main node] (3)  [right of=2]        {};  
		\node[main node] (4) 	[right of=3]        {};						 
		\node[main node] (5) 	[below of=1]		{};						 
		\node[main node] (6)  [right of=5]        {$v_k$};
		\node[main node] (7)  [right of=6]        {};  
		\node[main node] (8) 	[right of=7]        {};

		\draw [->] (2) edge (1) (3) edge (4) 
		(5) edge (6) (8) edge (7) ;
		
		\path[every node/.style={font=\sffamily\small}]		 
		(2) edge node [] {} (3)
		(7) edge node [] {} (6);
		
		\end{tikzpicture} \]
	\end{figure}
	\FloatBarrier

	First assume that $v_{k}$ is losing one chip in the initial firing. The net effect of the initial firing on $v_k$ is a decrease of one chip, and the net effect of the initial firing on $v_{k-1}$ is a decrease of one chip. This means that $|v_k|_0 = |v_{k-1}|_0$ and $|v_k|_0 - 1 = |v_{k-1}|_0 - 1$. Therefore, $|v_{k-1}|_0$ is the only possible initial stack size for $v_k$. Since there is only one possible initial stack size, $v_k$ has a multiplier of 1.
	
	Now assume that instead, $v_{k}$ is gaining one chip in the initial firing. The net effect of the initial firing on $v_k$ is an increase of one chip, and the net effect of the initial firing on $v_{k-1}$ is an increase of one chip. This means that $|v_k|_0 = |v_{k-1}|_0$ and $|v_k|_0 + 1 = |v_{k-1}|_0 + 1$. Therefore, $|v_{k-1}|_0$ is the only possible initial stack size for $v_k$. Since there is only one possible initial stack size, $v_k$ has a multiplier of 1.
	
	\textbf{Case 7:}

	
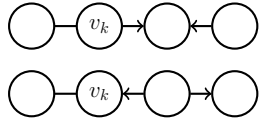
\begin{figure}[H]
		\[
		\begin{tikzpicture}[-,-=stealth', auto,node distance=1.5cm,
		thick,scale=0.6, main node/.style={scale=0.6,circle,draw, minimum size=1cm, font=\sffamily\Large\bfseries}]                               
		\node[main node] (1) 					    {};						 
		\node[main node] (2)  [right of=1]        {$v_k$};
		\node[main node] (3)  [right of=2]        {};  
		\node[main node] (4) 	[right of=3]        {};						 
		\node[main node] (5) 	[below of=1]		{};						 
		\node[main node] (6)  [right of=5]        {$v_k$};
		\node[main node] (7)  [right of=6]        {};  
		\node[main node] (8) 	[right of=7]        {};

		\draw [->] (2) edge (3) (4) edge (3) 
		(7) edge (6) (7) edge (8) ;
		
		\path[every node/.style={font=\sffamily\small}]		 
		(1) edge node [] {} (2)
		(5) edge node [] {} (6);
		
		\end{tikzpicture} \]
	\end{figure}
	\FloatBarrier

	First assume that $v_{k}$ is losing one chip in the initial firing. The net effect of the initial firing on $v_k$ is a decrease of one chip, and the net effect of the initial firing on $v_{k-1}$ is an increase of two chips. This means that $|v_k|_0 > |v_{k-1}|_0$ and $|v_k|_0 - 1 < |v_{k-1}|_0 + 2$.  So, $|v_{k-1}|_0 < |v_k|_0 < |v_{k-1}|_0 + 3$. Therefore, $|v_{k-1}|_0 + 1$ and $|v_{k-1}|_0 + 2$ are the only possible initial stack sizes for $v_k$. Since there are only one two possible initial stack sizes, $v_k$ has a multiplier of 2.
	
	Now assume that instead, $v_{k}$ is gaining one chip in the initial firing. The net effect of the initial firing on $v_k$ is an increase of one chip, and the net effect of the initial firing on $v_{k-1}$ is a decrease of two chips. This means that $|v_k|_0 < |v_{k-1}|_0$ and $|v_k|_0 + 1 > |v_{k-1}|_0 - 2$.  So, $|v_{k-1}|_0 > |v_k|_0 > |v_{k-1}|_0 - 3$. Therefore, $|v_{k-1}|_0 - 1$ and $|v_{k-1}|_0 - 2$ are the only possible initial stack sizes for $v_k$. Since there are only one two possible initial stack sizes, $v_k$ has a multiplier of 2.
\end{proof}

We now state a number of corollaries that come from the results regarding the multipliers of specific vertices found in Theorem~\ref{thm:multiplier}. In particular, these corollaries will allow us to break the problem of counting all $p_2$-configurations on $P_n$ into three cases: $p_2$-configurations that exist on alternating arrow orientations on $n$ vertices, $p_2$-configurations in which, moving from right to left, a flat edge appears before the first pair of adjacent agreeing edges, and $p_2$-configurations in which, moving from right to left, the first pair of adjacent agreeing edges appears before the first flat. We will then add up these three totals to determine the number of $p_2$-configurations that exist on $P_n$.  

\begin{corollary}\label{lem:an}
	The number of period configurations that exist on alternating arrow orientations on $P_n$ ($n \geq 3$) is $8 \times 3^{n-3}$. 
\end{corollary}

\begin{proof}
	In an alternating arrow orientation, every edge $e_i$ disagrees with the previous edge $e_{i-1}$. So, an alternating arrow orientation is unique based on the orientation of $e_1 = v_1v_2$. Therefore, there exist two alternating arrow orientations on a given path $P_n$, $n>1$.

	From Theorems~\ref{thm:littlemultiplier} and ~\ref{thm:multiplier}, we get that the multiplier for $v_1$ is 1, the multiplier for both $v_2$ and $v_n$ is 2, and every other multiplier is 3.

	Thus, the number of period configurations on a particular alternating arrow orientation on $P_n$ is $1 \times 2 \times 2\times 3^{n-3}$. Multiplying by two different alternating arrow orientations depending on the orientation of the first edge, we get that the number of period configurations that exist on alternating arrow orientations on $P_n$, $n \geq 3$, is  $1 \times 2 \times 2 \times 3^{n-3} \times 2 = 8 \times 3^{n-3}$.
\end{proof}

Define a sequence $A_n$ to represent the number of period configurations on an alternating path on $n$ vertices. $A_n={0,2,8,24,72,216,648,...,A_k,3A_k,3 \times 3A_k,...}$.

\begin{corollary}\label{cor:3an}
	For all $n \geq 3$, $3A_n=A_{n+1}$.
\end{corollary}

\begin{claim}
	Let $R$ be a $p_2$-orientation on $P_n$, $n \geq 2$. Let $R_1$, $R_2$, \dots, $R_k$, be the suborientations of $R$ on the $k$ disjoint paths created by removing $k-1$ flat edges from $P_n$. Then, for $1 \leq i \leq k$, $R_i$ is a $p_2$-orientation on its respective path.
\end{claim}

\begin{proof}
	Let $R$ be a period orientation on $P_n$, $n \geq 2$, with at least one flat edge. Let $R_1$ and $R_2$ be the suborientations of $R$ on the disjoint paths created by removing a single flat edge from $P_n$. 
	We will run through our checklist from Theorem~\ref{thm:illegalorientations} to determine whether or not these suborientations are themselves period orientations of their respective subgraphs. 
	
	\begin{enumerate}[(a)]
		
		\item Since there is no pair of adjacent flat edges in $R$, there cannot be a flat pair of adjacent edges in either $R_1$ or $R_2$.
		
		\item In both $R_1$ and $R_2$, there is an edge incident with a leaf that is not incident with a leaf in $R$. Call these edges $e_a$ and $e_b$. However, in $R$, we know that every edge incident with a flat edge must be directed (not flat). Thus, neither $e_a$ nor $e_b$ is flat.
		
		\item Since every flat edge in $R$ is incident with both a right and left edge, this is also true of $R_1$ and $R_2$.
		
		\item The removal of flat edges can have no effect on this rule for adjacent agreeing edges.
		
	\end{enumerate}
	
	So, we can conclude that $R_1$ and $R_2$, and thus, any number of disjoint orientations created by removing flats from a period orientation, are themselves, period orientations.
	
\end{proof}

\begin{corollary}\label{lem:sever}
	Let $R$ be a period orientation of $P_n$. Let $R_1$, $R_2$, \dots, $R_k$, be the suborientations of $R$ on the $k$ disjoint paths created by removing $k-1$ flat edges from $P_n$. The number of period configurations that exist on $R$ is equal to the product of the number of period configurations that exist on the suborientations $R_1$, $R_2$, \dots, $R_k$. 
\end{corollary}

\begin{proof}
	
	Let $R$ be a period orientation on $P_n$ with at least one flat edge. Let $R_1$ and $R_2$ be the suborientations of $R$ on the disjoint paths created by removing a single flat edge from $P_n$. 
	
	\begin{figure}[h]
		\[
		\begin{tikzpicture}[-,-=stealth', auto,node distance=1cm,
		thick,scale=0.6, main node/.style={scale=0.6,circle,draw, minimum size=1.2cm, font=\sffamily\Large\bfseries}]
		
		\node[main node] (1) 					    {$v_n$};						 
		\node[main node] (2)  [right=0.8cm of 1]        {$v_{n-1}$};
		\node[main node] (3)  [right=0.8cm of 2]        {$v_{n-2}$};  
		\node[main node] (4) 	[right=0.8cm of 3]        {$v_{k+2}$};						 
		\node[main node] (5)  [right=0.8cm of 4]        {$v_{k+1}$}; 
		\node[main node] (6) 	[right=0.8cm of 5]	    {$v_k$};						 
		\node[main node] (7)  [right=0.8cm of 6]        {$v_{k-1}$};
		\node[main node] (8)  [right=0.8cm of 7]        {$v_3$};  
		\node[main node] (9) 	[right=0.8cm of 8]        {$v_2$};						 
		\node[main node] (10) [right=0.8cm of 9]        {$v_1$}; 
		
		\path[every node/.style={font=\sffamily\small}]

		(2) edge (1) (2) edge (3) (5) edge (4)
		(6) edge (7) (8) edge (9)
		(10) edge (9);	
		
		\draw[dotted]
		(3) edge (4);
		
		\draw[dotted]
		(7) edge (8);
		\end{tikzpicture} \]
		\caption{$P_n$ with edge $v_kv_{k+1}$ removed.}
		\label{}
	\end{figure}
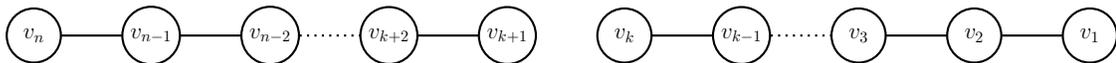 
	\FloatBarrier
	
	Suppose we removed just one flat edge, $v_kv_{k+1} = e_k$. The only vertices that could have an altered multiplier are those which are endpoints of $e_k$ or $e_{k+1}$. The vertices in question are $v_k$, $v_{k+1}$, and $v_{k+2}$. However, since what is measured in calculating the multiplier is the net effect of the firing, being incident to a flat edge is equivalent to not being incident to an edge at all. In particular, note that $v_{k+1}$, appearing to the left of the flat edge $e_k$, has only one possible initial stack size, that being $|v_k|_0$. This is equivalent to $v_{k+1}$ having only one possible initial stack size, by convention, when viewed as the right leaf in $R_2$. So, it follows that any number of flat edge removals will still maintain this result. 
	
\end{proof}

Next, we present a corollary of the multiplier theorem (Theorem~\ref{thm:multiplier}) which will be useful in determining the number of $p_2$-configurations that exist which induce orientations with adjacent agreeing arrows. It will be shown that, given a $p_2$-orientation of $P_{n}$ which contains some suborientation $v_{k+2}e_{k+1}v_{k+1}e_kv_k$ such that $e_{k+1}$ agrees with $e_k$, the orientation of $P_{n-2}$ created by contracting the edges $e_{k+1}$ and $e_k$ and reversing the direction of all directed edges $e_i$, $i > k+1$, is induced by the same number of $p_2$-configurations. We see an example of two such graph orientations in Figure~\ref{fig:edgecontraction}.

\begin{figure}[H]	
	\[
	\begin{tikzpicture}[-,-=stealth', auto,node distance=1.5cm,
	thick,scale=0.6, main node/.style={scale=0.6,circle,draw}]                               
	\node[main node] (1) 					    {$v_9$};						 
	\node[draw=none] (50) [left=1cm of 1] {$R$};
	\node[main node] (2)  [right=1cm of 1]        {$v_8$};
	\node[main node] (3)  [right=1cm of 2]        {$v_7$};  
	\node[main node] (4) 	[right=1cm of 3]        {$v_6$};						 
	\node[main node] (5)  [right=1cm of 4]        {$v_5$}; 
	\node[main node] (6)  [right=1cm of 5]        {$v_4$}; 
	\node[main node] (7) 	[right=1cm of 6]        {$v_3$};						 
	\node[main node] (8)  [right=1cm of 7]        {$v_2$}; 
	\node[main node] (9)  [right=1cm of 8]        {$v_1$};
	\node[main node] (10) [below=1.5cm of 1]        {$v_9$};
	\node[draw=none] (51) [left=1cm of 10] {$R'$};
	\node[main node] (11) [right=1cm of 10]        {$v_8$};  
	\node[main node] (12) [right=1cm of 11]        {$v_7$};						 
	\node[main node] (13) [below=1.5cm of 4]        {$v_6$}; 
	\node[main node] (14) [below=1.5cm of 7]        {$v_3$}; 
	\node[main node] (15) [right=1cm of 14]        {$v_2$};						 
	\node[main node] (16) [right=1cm of 15]        {$v_1$};

	\draw [->] 
	(1) edge node {$e_8$} (2) 
	(3) edge node [above] {$e_7$} (2) 
	(3) edge node {$e_6$} (4) 
	(5) edge node [above] {$e_5$} (4)
	(5) edge node {$e_4$} (6)
	(6) edge node {$e_3$} (7) 
	(8) edge node [above] {$e_2$} (7)
	(8) edge node {$e_1$} (9)
	(11) edge node {$e_8$} (10)
	(11) edge node [below] {$e_7$} (12)
	(13) edge node {$e_6$} (12)
	(13) edge node [below] {$e_5$} (14)
	(15) edge node {$e_2$} (14)
	(15) edge node [below] {$e_1$} (16);

	\end{tikzpicture} \]
	\caption{Graph orientations $R$ and $R'$, created by contracting two adjacent agreeing edges and reversing the direction of all subsequent directed edges}
	\label{fig:edgecontraction}
\end{figure}
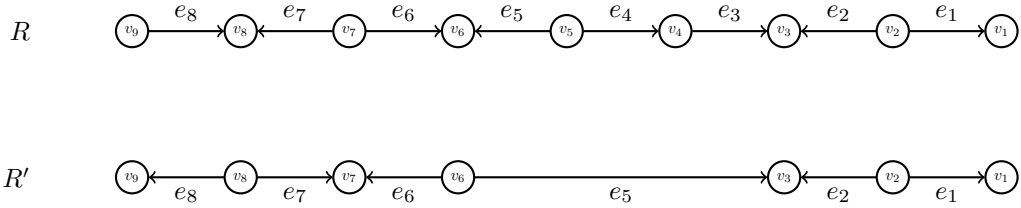

Note how the edges $e_3$ and $e_4$ have been contracted, removing $v_4$, and $v_5$, and every directed edge occurring to the left of the contraction has reversed direction. 

\begin{corollary}\label{cor:agreeingarrows}
	Suppose there exist adjacent agreeing edges $e_k=v_kv_{k+1}$ and \\ $e_{k+1}=v_{k+1}v_{k+2}$ in a period orientation, $R$, on a path, $P_n$, $n \geq 4$. Let $R'$ be the graph orientation created by contracting $e_k$ and $e_{k+1}$, reversing direction of every directed edge $e_i$, $i > k + 1$, and maintaining every other edge orientation from $R$. The number of $p_2$-configurations on $R$ is equal to the number of $p_2$-configurations on $R'$.
\end{corollary}

\begin{proof}

	By Theorem~\ref{thm:multiplier}, given a suborientation $v_{k+2}e_{k+1}v_{k+1}e_kv_k$ of a $p_2$-orientation on a path $P_n$ in which $e_{k}$ and $e_{k+1}$ agree, the multipliers of $v_{k+2}$ and $v_{k+1}$ are both equal to one. By contracting $e_k$ and $e_{k+1}$, we are removing $v_{k+2}$ and $v_{k+1}$ from the orientation. By removing these vertices, assuming every other multiplier has been maintained, the number of $p_2$-configurations that exist inducing the resulting orientation is the same as the number of $p_2$-configurations that exist inducing the original orientation. Due to the reversing direction of every subsequent directed edge, $v_k$ remains within the same $P_4$ suborientation from Theorem~\ref{thm:multiplier} and thus, maintains the same multiplier. Finally, every vertex appearing to the left of this contraction has had any incident directed edges reverse direction. However, by Theorem~\ref{thm:multiplier}, such a flipping of directed edges does not change a vertex's multiplier. Therefore, the product of multipliers must only be divided by $1 \times 1$ (the product of the multipliers of the removed vertices) to accommodate the edge contraction, and thus, the number of $p_2$-configurations does not change.

	So, we get that the number of configurations on $R$ is equal to the number of configurations on $R'$ and this process can be repeated until all pairs of adjacent agreeing edges have been removed.
	
\end{proof}

Let $T_n$ be the number of $p_2$-configurations that exist on $P_n$.

\begin{theorem}\label{thm:fn}
	For all paths $P_n$, $n \geq 4$, $$T_{n+4} = 3T_{n+3} + 2T_{n+2} + T_{n+1} - T_{n}$$ with $T_1 = 0$, $T_2 = 2$, $T_3 = 8$, and $T_4 = 26$. 
\end{theorem}

We will prove this theorem with the help of a number of claims.

In order to count the number of $p_2$-configurations on $P_n$, we will divide the set of all $p_2$-orientations into 3 cases. We have already solved for the number of $p_2$-configurations that exist on alternating arrow orientations on $n$ vertices, $A_n$. Our other two cases will be the case in which, moving from right to left, a flat edge appears before the first pair of adjacent agreeing edges, and the case in which, moving from right to left, the first pair of adjacent agreeing edges appears before the first flat. We will then add up these three totals to determine the number of $p_2$-configurations that exist on $P_n$.

\begin{claim}\label{clm:countedgefirst}
	The number of $p_2$-configurations on $P_n$, $n \geq 4$, in which, moving from right to left, a flat appears before the first pair of adjacent agreeing edges is $$\sum_{k=2}^{n-2} \frac{1}{2} A_k \times T_{n-k}$$
\end{claim}

\begin{proof}
	
	Suppose that, moving from right to left, the graph orientation, $R$, is alternating until the first flat appears. That is, the first flat appears before the first pair of adjacent agreeing edges appear. By Corollary~\ref{lem:sever}, the number of $p_2$-configurations that exist on a graph with some flat edge $e_k$ is equal to the product of the numbers of $p_2$-configurations that exist on the two suborientations created by removing $e_k$. Let $e_k = v_{k+1}v_k$ be the flat edge with the least index. We know that the suborientation $R_1 = v_ke_{k-1}v_{k-1}e_{k-2} \dots e_1v_1$ is an alternating arrow orientation by supposition. We know less about the suborientation $R_2 = v_ne_{n-1}v_{n-1}e_{n-2} \dots e_{k+1}v_{k+1}$. By Lemma~\ref{lem:rfl}, we know that $e_{k+1}$ disagrees with $e_{k-1}$. The number of configurations of $P_{n-k}$ which induce a graph orientation in which the orientation of the edge with the least index is given (without loss of generality, suppose it is right) is equal to $\frac{1}{2}F_{n-k}$ since half of the possibilities are excluded because the direction of the first edge is already known. So by Corollary~\ref{lem:sever}, the number of $p_2$-configurations which induce $R$ is $\frac{1}{2}F_{n-k} \times A_k$. Summing this value over all possible edges that could represent the first flat edge, we get
	
	$$\sum_{k=2}^{n-2} \frac{1}{2} A_k \times T_{n-k}$$
\end{proof}

\begin{definition}
	The \textbf{$k^{th}$ stage} of a path is the total number of $p_2$-configurations that exist on that path in which either a pair of adjacent agreeing edges or a flat appears within the first $k+1$ edges. 
\end{definition}

For example, on $P_8$, the $1^{st}$ stage is the number of period configurations that exist in which the second edge is flat. The $2^{nd}$ stage is the number of period configurations that exist in which either the second or third edge is flat, or the third edge agrees with the second edge. And the $3^{rd}$ stage is the number of period configurations that exist in which either the second, third, or fourth edge is flat, or either the third or fourth edge agrees with its previous edge. We denote the $k^{th}$ stage of $P_n$ by $\stage^k_n$.

\begin{claim}\label{clm:stage}
	If $n>2$, then $\stage^{k}_{n}=T_n$ for all $k \geq n-1$ and $\stage^{k}_{n}=T_n - A_n$ if $k = n-2$ or $k = n-3$.
\end{claim}

\begin{proof} 
	
	\textbf{Case 1:} \textit{$k=n-3$}
	
	From the definition of stage, we are counting the number of period configurations that exist on $P_n$ in which either a pair of adjacent agreeing edges or a flat appears within the first $n-2$ edges. The edge $e_{n-1}$ cannot be flat or agree with $e_{n-2}$ by Theorem~\ref{thm:illegalorientations}. So, $\stage^{n-3}_n$ counts every $p_2$-configuration except for those that induce an alternating arrow orientation. Thus, $\stage^{n-3}_n = T_n - A_n$.
	
	\textbf{Case 2:} \textit{$k=n-2$}
	
	We are counting every configuration from Case 1, but also including the possibility of $e_{n-1}$ being flat and the possibility of $e_{n-1}$ agreeing with $e_{n-2}$. However, by Theorem~\ref{thm:illegalorientations}, there are no $p_2$-orientations in which either of these situations arise. So, $\stage^{n-2}_n = \stage^{n-3}_n = T_n - A_n$.
	
	\textbf{Case 3:} \textit{$k \geq n-1$}
	
	We are counting every configuration from Case 2, but also including the possibility that we fail to find a flat edge or pair of adjacent agreeing edges within the $n-1$ edges. So, every $p_2$-configuration must be counted. So, $\stage^{k}_n = T_n$ for all $k \geq n-1$.

\end{proof}

\begin{claim}
	The number of $p_2$-configurations on $P_n$, $n \geq 5$, in which, moving from right to left, a pair of adjacent agreeing edges appear before a flat is $$\sum_{k=3}^{n-3} T_{n-2} - \stage^{k-2}_{n-2}.$$
\end{claim}

\begin{proof}
	We say $n \geq 5$ since, by Theorem~\ref{thm:illegalorientations}, no pair of adjacent agreeing edges can exist in a $p_2$-configuration on a path with fewer than 5 vertices. We are assuming that, moving from right to left, the graph orientation is entirely alternating until the first pair of adjacent agreeing edges appears. That is, the first pair of adjacent agreeing edges appears before the first flat appears. When adjacent agreeing edges appear, the number of $p_2$-configurations is equal to the number of $p_2$-configurations on the path with two fewer vertices in which the agreeing edges are removed and subsequent directed edges are reversed as outlined in Corollary~\ref{cor:agreeingarrows}. So every graph orientation of this form on $P_n$ can be viewed as a similar graph orientation on $P_{n-2}$ without changing the multipliers of any vertices. This allows for a recurrence, helping us to evaluate $F_n$ using $F_{n-2}$. However, we have supposed that up to some edge, $e_k$, the graph orientation is alternating. So, we must subtract the proper stage of the path on $n-2$ vertices. This will remove the possibility of agreeing edges and flat edges appearing to the right of $e_k$. Taking this sum over all possible edges that could represent, moving from right to left, the first edge that agrees with its immediate predecessor, we get
	
	$$\sum_{k=3}^{n-3} T_{n-2} - \stage^{k-2}_{n-2}$$
\end{proof}

We now calculate $T_n$ based on $T_{n-1}$, $T_{n-2}$, $T_{n-3}$, and $T_{n-4}$. Given a $p_2$-orientation on $P_n$, there are 4 mutually exclusive cases: $e_{n-2}$ is flat, $e_{n-3}$ is flat, $e_{n-2}$ and $e_{n-3}$ agree, $e_{n-2}$ and $e_{n-3}$ disagree. We know that these are the only possibilities by Theorem~\ref{thm:illegalorientations}.

For each of these four cases, we will determine the number of $p_2$-orientations that exist in that case. We then add up these four totals to calculate $T_n$. 

\textbf{Case 1:} $e_{n-2}$ is flat.

This calculation is equivalent to the first flat edge being $e_2$. We know that this is $A_2 \times \frac{1}{2} T_{n-2} = T_{n-2}$. This is the $k=2$ summand from Claim~\ref{clm:countedgefirst}.

\textbf{Case 2:} $e_{n-3}$ is flat.

This calculation is equivalent to the first flat edge being $e_3$. $A_3 \times \frac{1}{2} T_{n-3} = 4T_{n-3}$. This is the $k=3$ summand from Claim~\ref{clm:countedgefirst}.

\textbf{Case 3:} $e_{n-2}$ and $e_{n-3}$ agree.

We use our rule from Corollary~\ref{cor:agreeingarrows} for compacting agreeing arrows. What we get is every solution on $P_{n-2}$ that begins with two disagreeing arrows. This is equivalent to just subtracting the possibility that the first edge is flat. When $e_2$ is flat, we get $A_2 \times T_{n-4} = T_{n-4}$.  
So, we get $T_{n-2} - T_{n-4}$.

\textbf{Case 4:} $e_{n-2}$ and $e_{n-3}$ disagree.

This can be viewed as adding a new leftmost vertex to $P_{n-1}$. This vertex adds a multiplier of 3 (being amongst an alternating arrow suborientation) unless $e_{n-3}$ is flat. However, since we know $e_{n-3}$ to not be flat, we can exclude it from our calculation. If $e_{n-3}$ is flat in $P_{n-1}$, then there are $A_2 \times \frac{1}{2} T_{n-3} = T_{n-3}$ $p_2$-configurations. So, we get $3(T_{n-1} - T_{n-3})$.

The total sum is thus, $T_n = T_{n-2} + 4T_{n-3} + T_{n-2} - T_{n-4} + 3(T_{n-1} - T_{n-3}) = 3T_{n-1} + 2T_{n-2} + T_{n-3} - T_{n-4}$.

In order to find the explicit formula, we must perform some algebra:

\begin{align*}
T_n &= 3T_{n-1} + 2T_{n-2} + T_{n-3} - T_{n-4}\\
T_n - 3T_{n-1} - 2T_{n-2} - T_{n-3} + T_{n-4} &= 0 && \text{Let $T_n$ = $x^n$}\\
x^n - 3x^{n-1} - 2x^{n-2} - x^{n-3} + x^{n-4} &= 0\\
x^{n-4}(x^4 - 3x^{3} - 2x^{2} - x + 1) &= 0\\
x^4 - 3x^{3} - 2x^{2} - x + 1 &= 0
\end{align*}

The roots of this equation are $\alpha_1 \approx 3.6096$,
$\alpha_2 \approx 0.4290$, $\alpha_3 \approx -0.5193 - 0.6133i$, and $\alpha_4 \approx -0.5193 + 0.6133i$.

The solution for the $k^{th}$ value of this recurrence is

\vspace{1cm}

$$
\sum _{i=1}^{4} - {\frac { \left( -2 {{\alpha_i}}^{-2} -6 {\alpha_i}^{-1} + 2  \right) 
		\left( {{ \alpha_i}} \right) ^{k}}{ \left( 4\,{{ \alpha_i}}^{-3}-3\,{
			{ \alpha_i}}^{-2}-4\,{ \alpha_i}^{-1} - 3 \right) { \alpha_i}^{-1}}}$$

\vspace{1cm}

This can be rewritten as $T_k = c_1(\alpha_1)^{\;k} + c_2(\alpha_2)^{\;k} + c_3(\alpha_3)^{\;k} + c_4(\alpha_4
)^{\;k}$. The dominating term, out of these four roots, is the one which has the greatest modulus. These values are roughly 3.6096, 0.4290, 0.8036, and 0.8036. Thus, in the equation $T_k = c_1(\alpha_1)^{\;k} + c_2(\alpha_2)^{\;k} +c_3(\alpha_3)^{\;k} +c_4(\alpha_4)^{\;k}$, the dominant term is $c_1(\alpha_1 )^{\;k}\approx (0.1564)(3.6096)^{\;k}$.

\begin{corollary}\label{cor:pathconfigurationasymptotic}
	$T_k$ has an asymptotic value of $0.1564 \times 3.6096^k$.
\end{corollary}

Suppose now that a graph $G_k$ is composed of some graph $G_0$ connected to a path $P_k$, $k \geq 4$, with a bridge (an edge which, upon removal, would disconnect the graph). Due to the fact that multiplier calculations are localized for each vertex in a path, we conjecture that if a new vertex $v$ were added to the end of this path then the new vertex will be such a distance away from $G_0$ that this recurrence relation will hold. That is, if we know the number of $p_2$- configurations that exist on the four graphs $G_i$, $i=0,1,2,3$, then this same recurrence relation can calculate the number of $p_2$-configurations that exist on $G_4$. In this way, we conjecture that our recurrence relation solution extends to any graph connected to a path of length at least 4. 

\begin{conjecture}
	Let $G_k$ be a graph composed of some graph $G_0$ connected to a path $P_k$, $k \geq 3$, with a bridge. Then the number of $p_2$-configurations on $G_k$, $F(G_k)$, can be determined using the recurrence
	
	$$F(G_k) = 3F(G_{k-1}) + 2F(G_{k-2}) + F(G_{k-3}) - F(G_{k-4}). $$  
\end{conjecture}

\section{Conclusion}

This result for paths sits alongside similar results for complete graphs from ``On Variants of Diffusion", T. Mullen, PhD Thesis. When compared to the methods used on complete graphs with regards to polyominoes, the methods from this paper are rather crude. It is the opinion of the authors that results on other graph classes will not be so simple as those found here and in ``On Variants of Diffusion". We have plucked the low-hanging fruit and we believe that further results on the number of non-isomorphic configurations on specific graph classes will require computer data and may not result in (comparatively) nice third- and fourth-order recurrence relations.

\end{document}